\documentclass{article}
\usepackage{amsmath}
\usepackage{amsthm}
\usepackage{tikz}
\usepackage{amssymb}
\usepackage{mathtools}
\usepackage{mathabx}
\usepackage{algorithm}
\usepackage{algpseudocode}
\usepackage[toc]{appendix}
\usepackage[nameinlink]{cleveref}

\DeclareMathAlphabet{\mathbbold}{U}{bbold}{m}{n}
\usepackage[numbers,sort&compress]{natbib}
\usepackage[all]{xy}
\newtheorem{theorem}{Theorem}[section]
\newtheorem{lemma}[theorem]{Lemma}
\newtheorem{warning}[theorem]{Warning}
\newtheorem{conjecture}[theorem]{Conjecture}
\newtheorem{remark}[theorem]{Remark}
\newtheorem{corollary}[theorem]{Corollary}
\newtheorem{proposition}[theorem]{Proposition}

\theoremstyle{definition}
\newtheorem{definition}[theorem]{Definition}
\newtheorem{example}[theorem]{Example}

\theoremstyle{remark}

\usepackage{graphicx} 
\newcommand{\ZZ}{\mathbf{Z}}
\newcommand{\QQ}{\mathbf{Q}}
\newcommand{\RR}{\mathbf{R}}
\newcommand{\FF}{\mathbf{F}}
\newcommand{\CC}{\mathbf{C}}
\newcommand{\un}{\underline}
\newcommand{\Lie}{\mathrm{Lie}}
\newcommand{\val}{\mathrm{val}}
\newcommand{\sh}{\mathrm{sh}}
\newcommand{\orb}{\mathrm{orb}}
\newcommand{\an}{\mathrm{an}}
\newcommand{\op}{\mathrm{op}}
\newcommand{\cyc}{\mathrm{cyc}}
\newcommand{\Gal}{\mathrm{Gal}}
\newcommand{\R}{\mathrm{R}}
\newcommand{\Mod}{\mathrm{Mod}}
\newcommand{\Mat}{\mathrm{Mat}}
\newcommand{\la}{\mathrm{la}}
\newcommand{\LT}{\mathrm{LT}}

\newcommand{\A}{\mathrm{A}}

\newcommand{\han}{\text{-an}}
\newcommand{\hla}{\text{-la}}
\newcommand{\pa}{\mathrm{pa}}
\newcommand{\WLambda}{\widetilde{\Lambda}}
\newcommand{\wt}{\widetilde}

\title{Locally analytic vectors and decompletion in mixed characteristic}
\author{Gal Porat
\\{galporat1@gmail.com}}
\begin{document}

\maketitle

\begin{abstract}
In $p$-adic Hodge theory and the $p$-adic Langlands program, Banach spaces with $\QQ_p$-coefficients and $p$-adic Lie group actions are central. Studying the subrepresentation of $G$-locally analytic vectors, $W^{\la}$, is useful because $W^{\la}$ can be studied via the Lie algebra $\Lie(G)$, which simplifies the action of $G$. Additionally, $W^{\la}$ often behaves as a decompletion of $W$, making it closer to an algebraic or geometric object.

This article introduces a notion of locally analytic vectors for $W$ in a mixed characteristic setting, specifically for $\ZZ_p$-Tate algebras. This generalization encompasses the classical definition and also specializes to super-Hölder vectors in characteristic $p$. Using binomial expansions instead of Taylor series, this new definition bridges locally analytic vectors in characteristic $0$ and characteristic $p$.

Our main theorem shows that under certain conditions, the map $W \mapsto W^{\la}$ acts as a descent, and the derived locally analytic vectors $\R_{\la}^i(W)$ vanish for $i \geq 1$. This result extends Theorem C of \cite{Po24}, providing new tools for propagating information about locally analytic vectors from characteristic $0$ to characteristic $p$.

We provide three applications: a new proof of Berger-Rozensztajn's main result using characteristic $0$ methods, the introduction of an integral multivariable ring $\widetilde{\mathbf{A}}_{\LT}^{\dagger,\la}$ in the Lubin-Tate setting, and a novel interpretation of the classical Cohen ring ${\bf{A}}_{\QQ_p}$ from the theory of $(\varphi,\Gamma)$-modules in terms of locally analytic vectors. 
\end{abstract}

\maketitle

\tableofcontents

\section{Introduction}

In $p$-adic Hodge theory and in the $p$-adic Langlands program, one often encounters a Banach space $W$ with $\QQ_p$-coefficients endowed with an action of a $p$-adic Lie group $G$. It has proven useful to study the subrepresentation of $G$-locally analytic vectors $W^{\la} \subset W.$ One reason for this is that $W^{\la}$ has an action of the Lie algebra $\Lie(G),$ which is often simpler to study than the action of $G$. Another reason for this usefulness is that $W^{\la}$ sometimes behaves as a decompletion of $W$ which is then closer to being of an algebraic  or geometric nature. 

As a classical example we consider Sen theory. Let $\QQ_p^{\cyc}$ denote the cyclotomic extension obtained by adding all $p$-power roots of unity to $\QQ_p$, $K$ be a finite extension of $\QQ_p$, $K^\cyc = K\QQ_p^{\cyc}$, $H_K = \Gal(\overline{K}/K^\cyc)$ and $G = \Gal(K^{\cyc}/K)$ which is isomorphic to an open subgroup of $\ZZ_p^\times$. We also let $\CC_p$ be the completion of $\overline{\QQ}_p$. If $X$ is a smooth proper variety over $K$ let $V = \mathrm{H}^{n}_{\text{\'et}}(X_{\overline{K}},\QQ_p)$. Then $V$ is a representation of $\Gal(\overline{K}/K)$, and we set
$$W = (\CC_p \otimes_{\QQ_p} V)^{H_K}.$$
It can then be shown (see for example Théorème 3.4 of \cite{BC16}) that $W^{\la}$ is a vector space over $K^{\cyc}$ to which $W$ descends. Furthermore, the action of $\Lie(G) \cong \ZZ_p$ gives a linear operator, called the Sen operator, which acts on $W^{\la}$. Its eigenvalues belong to $\ZZ$ and as a set are equal to the negatives of the Hodge numbers $\{h^{i,j}\}_{i+j=n}$ of $X$.

One issue with the classical definition of $G$-locally analytic vectors is that it is available only in the setting of $\QQ_p$-coefficients, while one would want to consider spaces $W$ with $\ZZ_p$ and $\FF_p$-coefficients as well. Nevertheless, it is clear that some phenomenon of this kind exists. For example, a Sen operator in a mixed characteristic was recently  introduced in work of Bhatt-Lurie (\cite{BL22}). It is unclear how to interpret it in terms of locally analytic vectors.
In another direction, Berger and Rozensztajn introduce in 
(\cite{BR22},\cite{BR24}) the notion of super-Hölder vectors, which serve as an analogue of locally analytic vectors in characteristic $p$. 

In this article, we focus on the notion of locally analytic vectors for $W$ with coefficients in a mixed characteristic setting, namely for $\ZZ_p$-Tate algebras. This notion generalizes the classical notion and specializes to the super-Hölder vectors notion of Berger and Rozensztajn in the characteristic $p$ setting. The basic idea is to use binomial expansions rather than Taylor series expansions to define analytic functions. This idea is well known to the experts, introduced for example in \cite{Gu19}, \cite{JN19} and \cite{BR22}. One then defines a locally analytic vector to be an element $w$ of $W$ whose associated orbit map $\mathrm{orb}_{w}:G\rightarrow W$ is locally analytic.

Why use binomial expansions instead of Taylor series expansions? First, we lose nothing by doing this: in characteristic 0, it is known by the Amice theorem \cite{Am64} that locally analytic functions are the same as those with exponentially decreasing binomial expansions. Second, new analytic functions can be produced in mixed characteristic and characteristic $p$. This can be motivated by the example $R = \FF_p((X))$ with the action of $\ZZ_p$ given by $a(X)=(1+X)^a$ (as considered in \cite{BR22}). We would like to consider such a function as analytic. Writing $a\mapsto a(X) = \sum_{n \geq 0}\binom{a}{n}X^n$, we see it has a binomial expansion in $a$ with the coefficients $X^n$ tending to $0$ exponentially in the topology of $\FF_p((X)).$ On the other hand, it does not have a Taylor series expansion: indeed, in characteristic $p$, it is easy to see that Taylor series expansions only give rise to locally constant functions on $\ZZ_p^d,$ since they must factor through $\FF_p^d.$  

One nice feature of this new definition is that it allows to link locally analytic vectors in characteristic $0$ and characteristic $p$. A template for this situation is as follows: we are given a module $A$ with coefficients in a $\ZZ_p$-Tate algebra so that $\overline{A} = A/p$ is also a $\ZZ_p$-Tate algebra, together with an injection of $A[1/p]$ into a $\QQ_p$-algebra $B$. If $A$ and $B$ are endowed with $G$-actions, we have maps $A^{\la} \rightarrow B^{\la}$ and $A^{\la} \rightarrow \overline{A}^{\la}$. Now $B^{\la}$ is a $\QQ_p$-vector space and thus can be studied via $p$-adic analysis and the action of $\Lie(G).$ If the map $A^{\la} \rightarrow \overline{A}^{\la}$ is sufficiently well behaved (for example, if it is surjective), one can use the inclusion of $A^{\la}\subset B^{\la}$ to study $\overline{A}^{\la}$ and thus study the analytic vectors in a characteristic $p$ object using $p$-adic analysis in characteristic $0$! A template example for this would be to take $A=\ZZ_p[[X^{1/p^\infty}]]\langle p/X\rangle[1/X],$ the ring of functions of a preperfectoid pseudorigid disc, $\overline{A} = \FF_p((X^{1/p^\infty}))$ and $B=\ZZ_p[[X^{1/p^\infty}]]\langle p/X,X/p\rangle[1/X]=\ZZ_p[[X^{1/p^\infty}]]\langle p/X,X/p\rangle[1/p]$ the ring of functions on a preperfectoid annulus, with the action of $\ZZ_p$ given by $a(X)=(1+X)^a.$ The key point is that topology on $B$ is both $X$-adic and $p$-adic, so it can be linked to $A$ on the one hand, and studied via $p$-adic analysis of $B^\la$ on the other. A very similar example will be studied in our first application in $\mathsection{4}.$

The main goal of this paper is to analyze the extent to which $W\mapsto W^{\la}$ acts as a decompletion. We suppose $W$ is a finite free module over a $\ZZ_p$-Tate algebra $R$ endowed with a semilinear action of a compact $p$-adic Lie group $G$. The main result of this article is the following.

\begin{theorem}
\label{thm: main_thm_intro}
Suppose that the ring $R$ satisfies the Tate-Sen axioms (TS1)-(TS4) (see $\mathsection{3.2}$). 

Then:

1. The natural map $$R \otimes_{R^{\la}} W^{\la} \rightarrow W$$ is an isomorphism.

2. If moreover the Lie algebra $\mathrm{Lie}(G)$ is abelian, the derived locally analytic vectors $\R_{\la}^i(W)$ of $W$ are $0$ for $i \geq 1$.
\end{theorem}

This theorem is a generalization of Theorem C of \cite{Po24}. The basic idea of the proof is in essence to take a "fiber product" of the original method of \cite{Po24} and the mixed characteristic results of \cite{Po22b}. As Theorem C of \cite{Po24} has been applied in several subsequent articles, we hope the current result will be useful in a similar way.

Note that the vanishing of $\R_{\la}^1(W)$ is precisely the type of behavior that will allow us to propagate information regarding locally analytic vectors along reductions. Indeed, if $W$ lives over some mixed characteristic $\ZZ_p$-Tate algebra such as $\ZZ_p[[X]]\langle p/X\rangle[1/X]$, we have also the $\FF_p$-Banach space $\overline{W}$ and an exact sequence
$$0\rightarrow W \xrightarrow{p} W\rightarrow \overline{W} \rightarrow 0.$$
After passing to locally analytic vectors, the vanishing of $\R_{\la}^1(W)$ implies that $W^\la\rightarrow\overline{W}^\la$ is surjective. 

We give three applications of our results. Our first result is a new proof of the main result of \cite{BR22}. There, Berger and Rozensztajn show that for the action of $G = \ZZ_p^\times$ on the $X$-adic completion of $\cup_{n} \FF_p((X^{1/p^n}))$, taking $G$-locally analytic vectors undoes the $X$-adic completion. We show how to deduce this using characteristic 0 methods, as sketched above. Our second application is the introduction of an integral multivariable ring $\widetilde{\mathbf{A}}_{\LT}^{\dagger,\la}$ in the Lubin-Tate setting, together with an appropriate overconvergence result for $(\varphi,\Gamma)$-modules. Previously, such a ring was only available rationally with pro-analytic vectors (as in \cite{Be16}). Our third and final result is a surprising description of the classical Cohen ring ${\bf{A}}_{\QQ_p}$ appearing in the theory of $(\varphi,\Gamma)$-modules in terms of locally analytic vectors. We show how to endow the ring $\wt{\bf{A}}_{\QQ_p}$ of Witt vectors of $\widehat{{\bf{\QQ}}}_p^\cyc$ with the structure of a mixed-characteristic Fréchet space, and its ring of locally analytic vectors is $\varphi^{-\infty}({\bf{A}}_{\QQ_p}).$ This reveals an analogy between $\varphi^{-\infty}({\bf{A}}_{\QQ_p})$ and the ring $\cup_n \QQ_p(\zeta_{p^n})[[t]]$ of locally analytic elements in $\bf{B}_{\mathrm{dR}}^+$, and suggests a way to give a lift of the field of norms beyond the Lubin-Tate setting, a question studied by several authors and motivated by Iwasawa theory (see for instance \cite{Be14}, \cite{Poy22}).

\subsection{Structure of the article}
In $\mathsection{2}$ we give the definitions and basic results regarding locally analytic vectors and functions in mixed characteristic. In $\mathsection{3}$ we prove the main result of the article. The reader may desire to skip this section if they are only interested in using the main result. Finally, in $\mathsection{4}$ we give three applications of our methods.

\subsection{Notations and conventions}

\subsubsection{Valuations}
By a valuation on a ring $R$, we mean a map $\val_R:R \rightarrow (-\infty, \infty] $ satisfying
the following properties for $x, y \in R$:

(1) $\val_R(x) = \infty$ if and only if $x = 0$ (i.e. $R$ is separated);

(2) $\val_R(xy) \geq \val_R(x) + \val_R(y)$;

(3) $\val_R(x+y) \geq \min(\val_R(x), \val_R(y))$.

This definition naturally extends to valuations
on an $R$-module $M$. We set $M^+ := M^{\val_M \geq 0}.$ If $f:X\rightarrow M$ is a function, we set $\val_M^{\op}(f):=\inf_{x\in X}\val_M(f(x))$ for the operator valuation. 

\subsubsection{$\ZZ_p$-Tate algebras}
A $\ZZ_p$-Tate algebra is a Tate ring (see Definition 2.2.5 of \cite{SW20}) which is a $\ZZ_p$-algebra, such that the map $\ZZ_p\rightarrow R$ is continuous. We can always endow such a ring with a valuation $\val_R$ by taking some topologically nilpotent unit $\varpi \in R$ and letting $\val_R(x)=\inf_n\{\varpi^nx\in R^+\}$ for some open subring $R^+\subset R$. However, it sometimes happens that there is a more natural valuation inducing the same topology which does not arise in this fashion, such as the $X$-adic valuation on $\FF_p((X^{1/p^\infty}))$. In any case, whenever $R$ is a $\ZZ_p$-Tate algebra endowed with a valuation $\val_R$, we shall always assume:

(1) There is a topologically nilpotent unit $\varpi \in R$ such that for any $R$-module endowed with a valuation we have $\val_M(\varpi x)=\val_R(\varpi)+\val_M(x)$ for $x\in M$. In particular, we always have $\val_R(1) = 0$, and

(2) We have the inequality\footnote{This condition is not automatic for a $\ZZ_p$-Tate algebra, as can be seen in the example  $R=\ZZ_p[[\varpi]]\langle{p^2/\varpi}\rangle[1/\varpi]$.} $\val_R(p) > 0$. 

\subsubsection{Miscellaneous}

 Given $\un{n} \in \ZZ^d$ we write $|\un{n}|_\infty=\max_{1 \leq i \leq d} |n_i|$ and $|\un{n}|=\sum_{1\leq i \leq d}|n_i|$.  For a real number $a$ we denote by $\lfloor a \rfloor$ the largest integer which is smaller or equal to $a$. Finally, for ${\un{x}} \in \RR^d$ we let $\lfloor \un{x} \rfloor=\sum_{i=1}^d \lfloor x_i \rfloor.$

\subsection{Acknowledgements}
I would like to thank Chengyang Bao, Laurent Berger, Hui Gao, Léo Poyeton and the anonymous referee for their helpful comments. This work was supported by the European Research Council (ERC 101078157).
\section{Locally analytic functions and vectors in mixed characteristic}
In this section we introduce the spaces of functions and analytic vectors that will appear in this article. 

\subsection{Locally analytic functions on $\ZZ_p^d$}

Let $R$ be a $\ZZ_p$-Tate algebra with valuation $\val_R$ inducing its topology (recall our conventions regarding valuations in $\mathsection{1.2}$). Let $M$ be an $R$-module endowed with a compatible valuation $\val_M$.

In particular, when $M = R$, we have the function from $\ZZ_p^d$ to $R$ given by $$\un{x}=(x_1,...,x_d) \mapsto \binom{\un{x}}{\un{n}} := \binom{x_1}{n_1}\cdot ... \cdot \binom{x_d}{n_d}.$$

Given a function $f:\ZZ_p^d \rightarrow M$ and an element $\un{y} \in \ZZ_p^d$, we write $\Delta_{\un{y}}(f):\ZZ_p^d \rightarrow M$ for the function given by $\Delta_{\un{y}}(f)(\un{x}) = f(\un{x}+\un{y})-f(\un{x})$. The operators $\Delta_{\un{y}}$ and $\Delta_{\un{z}}$ are commuting for $\un{y},\un{z}\in \ZZ_p^d$. We set $\Delta^n_{\un{y}}(f)$ for the $n$'th application of $\Delta_{\un{y}}$ to $f$. Given a $d$-dimensional vector $\un{n}$, we write $\Delta^{\un{n}}=\Delta_{1}^{n_1}\circ...\circ \Delta_{d}^{n_d}$, where $\Delta_{k} = \Delta_{1_k}$ for $1_k \in \ZZ_p^d$ being 1 on the $k$'th copy of $\ZZ_p$ and $0$ elsewhere. Setting $(-1)^{\un{m}} = (-1)^{m_1+...+m_d}$, it can be checked that
$$\Delta^{\un{n}}(f)(\un{x}) = \sum_{\un{i}=\un{0}}^{\un{n}} (-1)^{\un{n}+\un{i}} \binom{\un{n}}{\un{i}} f(\un{i}+\un{x})$$

and that for $\un{m} \leq \un{n}$ one has $\Delta^{\un{m}}(\binom{\un{x}}{\un{n}}) = \binom{\un{x}}{\un{n}-\un{m}}.$

For this section it will also be important to introduce for $\un{z} \in \ZZ_p^d$ the shift operator $\sh_{\un{z}}$ given by $\sh_{\un{z}}(f)(\un{x}) = f(\un{z}+\un{x})$. It commutes with the $\Delta_{\un{y}}$, and we have the identity 
\begin{align}
\label{identity:delta_sum}
    \Delta_{\un{x}+\un{y}}(f) = (\sh_{\un{y}}\circ\Delta_{\un{x}})(f) + \Delta_{\un{y}}(f).
\end{align} Recall $\mathsection{1.2}$ for the notation $\val_M^{\op}.$

\begin{lemma}
\label{lem:val_bound_gen_elem}
Let $\un{y} \in \ZZ_p^d$. Then for every $f:\ZZ_p^d \rightarrow M$ we have $$\val_{M}^{\op}({\Delta_{\un{y}}^n}(f)) \geq \min_{1\leq i \leq d}(\val_{M}^{\op}(\Delta_i^{\lceil n/d \rceil}(f))).$$ 
\end{lemma}

\begin{proof}
Using continuity, we may restrict to the case $\un{y} \in \ZZ^d_{\geq 0}$. By writing $\un{y}$ as a sum of $1_k$'s and applying the identity (\ref{identity:delta_sum}) inductively, we may write $$\Delta_{\un{y}} = \sum_j{a_j \sh_{\un{z}_j}} \Delta_{i_j}$$ for some $a_j \in \ZZ$, $\un{z}_j \in \ZZ^d$ and $1 \leq i_j \leq d$ (the $i_j$ may repeat). Raising this to the $n$'th power, we obtain a writing of $\Delta_{\un{y}}$ as a $\ZZ$-linear combination up to shifts of $\Delta^{\un{m}}$'s where for each $\un{m}$ we have $\sum{m_i}=n$. For each such $\un{m}$, choose the maximal coordinate $m_i$. For every $g:\ZZ_p \rightarrow M$ one has $\val_M^{\op}(\Delta^{\un{m}}(g))\geq\val_M^{\op}(\Delta_i^{m_i}(g))$; applying this observation to shifts of $f$ and using that $m_i\geq \lceil n/d \rceil$, this concludes the proof.
\end{proof}

We write $\mathcal{C}^0(\ZZ_p^d,M)$ for the continuous functions from $\ZZ_p$ to $M$. The following version of Mahler's theorem shows each continuous function has an expansion in terms of the $\binom{\un{x}}{\un{n}}$.

\begin{theorem}
\label{thm:Mahler}
The following are equivalent for a function $f:\ZZ_p^d \rightarrow M$.

1. We have $f \in \mathcal{C}^0(\ZZ_p^d,M).$

2. The function $f$ has an expansion of the form $f(\un{x}) = \sum_{\un{n}\in\ZZ_{\geq 0}^d} a_{\un{n}}(f) \binom{\un{x}}{\un{n}}$ with $\val_M(a_{\un{n}}(f)) \rightarrow \infty$ as $|\un{n}| \rightarrow \infty$.

Furthermore, under these equivalent conditions we have $a_{\un{n}} = \Delta^{\un{n}}(f)(\un{0})$ and $\val_M^{\op}(f) = \inf_{\un{n}\in \ZZ_{\geq 0}}\val_M(a_{\un{n}}). $

\end{theorem}

\begin{proof}
We follow the proof of Bojanic \cite{Bo74}, see also Theorem 1.13 of \cite{BR22}.

First, suppose $f$ has an expansion of the form $f(\un{x}) = \sum{a_{\underline{n}}}\binom{x}{\underline{n}}$ with $a_{\underline{n}} \rightarrow 0$. Then $f$ is continuous since it is the uniform limit of continuous functions. We now explain why do we have in this case the equality $$\val_M^{\op}(f) = \inf_{\un{n}\in \ZZ_{\geq 0}}\val_M(a_{\un{n}}).$$

Indeed, for each $\un{n}$, the element $a_{\un{n}}=\Delta^{\un{n}}(f)(\un{0})$ is a sum of elements in $\mathrm{Im}(f)$, which shows $\inf_{\un{x}}\val_M(f(\un{x})) \leq \inf_{\un{n}}\val_M(a_{\un{n}}).$ The inequality in other direction is clear given the existence of the expansion.

Conversely, suppose $f$ is continuous. We set $a_{\un{n}} = \Delta^{\un{n}}(f)(0)$. It suffices to show that $a_{\underline{n}} \rightarrow 0$; indeed, with that given, the function $f$ and the one given by $g(\un{x}):=\sum_{\un{n}}a_{\underline{n}}\binom{x}{\underline{n}}$ agree on $\ZZ_{\geq 0}^d$, hence are equal.

Since $\ZZ_p^d$ is compact, the function $f$ is uniformly continuous.  Fix $s \geq 0$; then there exists $t \geq 0$ such that if $\val(\un{x}-\un{y}) \geq t$ then $\val_M(f(\un{x})-f(\un{y})) \geq s\val(p) + \val_{M}^{\op}(f)$. The same statement, with the same $t$, holds for any $\Delta^{\underline{n}}(f)$. To conclude the proof, it suffices to show that if $|\underline{n}|_{\infty}$ is larger than $sp^t$ then $\val(a_{\un{n}})\geq s\val(p)+\val_M^{\op}(f)$.

Now given a function $g$ we have $$\Delta^{p^t1_k}(g)(\un{x}) = g(\un{x}+p^t1_i)-g(\un{x}) + p\cdot y(\un{x})$$ for some element $y(\un{x})$ with $\val_M(y(\un{x})) \geq \val_M^{\op}(g)$.
Taking $g$ to be a function of the form $\Delta^{\underline{n}}(f)$, we obtain $\val(g(\un{x}+p^t1_i)-g(\un{x})) \geq s\val(p) + \val_{M}^{\op}(f)$ and so for every $\un{n}$ and $k$ we have the inequality 
\begin{align}
\label{ineq:val_op_low_bound}\text{ }\val_{M}^{\op}(\Delta^{p^t1_k}(\Delta^{\underline{n}}(f))) \geq \min(s\val(p) + \val_{M}^{\op}(f), \val(p) + \val_M^{\op}(\Delta^{\underline{n}}(f))).
\end{align}

Now if $|\underline{n}|_{\infty} > sp^t$ then the $k$'th coordinate of $\un{n}$ is larger than $sp^t$ for some $k$. Applying (\ref{ineq:val_op_low_bound}) consecutively to $$\Delta^{\underline{n}-sp^t1_k}(f),\Delta^{\underline{n}-(s-1)p^t1_k}(f),...,\Delta^{\underline{n}-p^t 1_k}(f)$$ shows $\val_{M}^{\op}(\Delta^{\underline{n}}(f)),$ hence also $a_{\un{n}}(f)$, has valuation at least $s\val(p) + \val_{M}^{\op}(f)$. This finishes the proof. 
\end{proof}

We now give several definitions for analytic functions on $\ZZ_p^d$ and explain the relationships between them. See also \cite{BR22}, \cite{BR24}, $\mathsection{3}$ of \cite{Gu19} and $\mathsection{3.2}$ of \cite{JN19} for similar definitions.

\begin{proposition}
\label{prop:equiv_conds_analytic_func}
Let $\lambda, \mu \in \RR$. The following are equivalent for a function $f:\ZZ_p^d \rightarrow M$.

1. The inequality $\val_M(a_{\un{n}}(f)) \geq p^\lambda \cdot p^{\lfloor \log_p(|\un{n}|_\infty) \rfloor} + \mu$ holds for every $\un{n} \in \ZZ^d_{\geq 0}$.

2. The inequality $\val_M^{\op}(\Delta^{n}_i(f)) \geq p^\lambda \cdot p^{\lfloor \log_p(n) \rfloor} + \mu$ holds for every $1\leq i \leq d$ and $n \in \ZZ_{\geq 0}.$
\end{proposition}

\begin{proof}
To show 1 implies 2, we may assume $n=p^k$. Write $f(\un{x}) =\sum a_{\un{n}} \binom{\un{x}}{\un{n}}$ (as we may according to \ref{thm:Mahler}), then
$$\Delta_i^{p^k}(f)(\un{x}) = \sum a_{\un{n}} \Delta_i^{p^k}(\binom{\un{x}}{\un{n}}) 
 = \sum_{\un{n}:n_i\geq p^k} a_{\un{n}}\binom{\un{x}}{\un{n}-1_ip^k},$$
 for which each term has $\val\geq p^\lambda \cdot p^k + \mu$ by assumption.

 Conversely, choose $i$ with $n_i = |\un{n}|_\infty$. Let $\un{n}'$ be $\un{n}$ with the $i$'th coordinate removed. Then
 $$a_{\un{n}}(f) = \Delta^{\un{n}}(f)(\un{0}) =(\Delta_i^{n_i}(\Delta^{\un{n}'}(f))(\un{0})).$$

 As $\Delta^{\un{n}'}(f)$ is a $\ZZ$-linear combination of shifts of $f$, we obtain that $a_{\un{n}}(f)$ is a $\ZZ$-linear combination of $\Delta_i^{n_i}(f)(t_j)$ for various $t_j$. Each of these has valuation $\geq p^{\lfloor \log_p(n_i) \rfloor} + \mu=p^{\lfloor \log_p(|\un{n}|_\infty) \rfloor} + \mu,$ as required.
\end{proof}

\begin{definition} 
\label{def:analytic_funcs}
1. Let $\mathcal{C}^{\an\text{-}\lambda,\mu}(\ZZ_p^d,M)$ be the submodule of functions satisfying any of the equivalent conditions of Proposition \ref{prop:equiv_conds_analytic_func}. By condition 1 and Theorem \ref{thm:Mahler}, it is contained in $\mathcal{C}^{0}(\ZZ_p^d,M)$.

2. Let $\mathcal{C}^{\lambda\text{-}\an}(\ZZ_p^d,M)\subset \mathcal{C}^0(\ZZ_p^d,M)$ be the submodule of functions $f$ satisfying $\val_M(a_{\un{n}}(f))-p^\lambda|\un{n}|\rightarrow \infty$. We endow it with the valuation $$\val_{\lambda}(f) := \inf_{\un{n}}(a_{\un{n}}(f)-\lfloor p^\lambda \un{n} \rfloor)$$ (recall $\mathsection{1.2}$ for the notation $\lfloor p^\lambda \un{n} \rfloor$.)
\end{definition}

\begin{proposition}
\label{prop:shifts_action} The modules $\mathcal{C}^{\an\text{-}\lambda,\mu}(\ZZ_p^d,M)$ and $\mathcal{C}^{\lambda\text{-}\an}(\ZZ_p^d,M)$ are stable under the shift operator $\sh_{\un{z}}$ for any $\un{z}\in \ZZ_p^d$.

Furthermore, the action of $\ZZ_p^d$ on $\mathcal{C}^{\lambda\text{-}\an}(\ZZ_p^d,M)$ given by shifts is continuous and each element of $\ZZ_p^d$ acts as an isometry.
\end{proposition}

\begin{proof}
For $\mathcal{C}^{\an\text{-}\lambda,\mu}(\ZZ_p^d,M)$ the statement follows from Proposition \ref{prop:equiv_conds_analytic_func}. We now show the statement holds for $\mathcal{C}^{\lambda\text{-}\an}(\ZZ_p^d,M)$. Suppose it were known for shifts by elements of $\ZZ^d$. Let $\un{z} \in \ZZ_p^d$ and let $f\in \mathcal{C}^{\lambda\text{-}\an}(\ZZ_p^d,M)$. As $f$ is uniformly continuous, we can find some $m$ so that $\val_M(f(\un{x}+\un{y})-f(\un{x}) > 1$ if $\val(\un{x}-\un{y}) > m$. Choosing $\un{z}_0 \in \ZZ^d$ with $\val(\un{z}-\un{z}_0) > m$, we see that $\val_M^{\op}(\sh_{\un{z}}(f)-\sh_{\un{z}_0}(f)) > 1$ and $\sh_{\un{z}_0}(f) \in \mathcal{C}^{\lambda\text{-}\an}(\ZZ_p^d,M)$ which implies $\sh_{\un{z}}(f) \in \mathcal{C}^{\lambda\text{-}\an}(\ZZ_p^d,M)$. This argument also shows the action of $\ZZ_p^d$ on $\mathcal{C}^{\lambda\text{-}\an}(\ZZ_p^d,M)$ is continuous (provided we show it exists). We thus reduce to the case where $\un{z} \in \ZZ^d$, which allows us to reduce further to the case $\un{z}=1_i$ for $1 \leq i \leq d$. Finally, we have $a_{\un{n}}(\sh_{1_i}(f))=a_{\un{n}}(f)+a_{\un{n}+1_i}(f)$ which shows $\sh_{1_i}(f) \in \mathcal{C}^{\lambda\text{-}\an}(\ZZ_p^d,M)$.

It remains to show each element of $\ZZ_p^d$ acts as an isometry on $\mathcal{C}^{\lambda\text{-}\an}(\ZZ_p^d,M)$. We may reduce again to showing that $\sh_{1_i}$ acts as an isometry. To show this, take $f \in \mathcal{C}^{\lambda\text{-}\an}(\ZZ_p^d,M)$ and let $\un{n}_0$ be maximal with $\val_\lambda(f) = \val_M(a_{\un{n}_0}(f))-\lfloor p^\lambda \un{n}_0 \rfloor)$. Then $\val_M(a_{\un{n}_0+1_i}(f)) > \val_M(a_{\un{n}_0}(f))$, and so the formula $a_{\un{n}}(\sh_{1_i}(f))=a_{\un{n}}(f)+a_{\un{n}+1_i}(f)$ shows $\val_{\lambda}(\sh_{1_i}(f)) \leq \val_\lambda (f).$ Now by using the uniform continuity again, we know that for $p^k$ for $k$ sufficiently large we have $\val_{\lambda}(\sh_{p^k 1_i}(f)) = \val_\lambda (f)$. Hence the inequalities
$$\val_{\lambda}(f) = \val_{\lambda}(\sh_{p^k 1_i}(f))  \leq \val_{\lambda}(\sh_{(p^k-1) 1_i}(f)) 
 \leq... \leq \val_{\lambda}(\sh_{1_i}(f)) \leq \val_{\lambda}(f)$$
 are all forced to be equalities, which shows $\val_{\lambda}(\sh_{1_i}(f)) = \val_\lambda (f),$ as required.
\end{proof}




\begin{lemma}
\label{lem:make_funcs_smaller_by_enlargning_l}
Let $\lambda \in \RR$ and $c > 0.$ Then there exist $l \geq 0$ such that for every $f \in \mathcal{C}^{\lambda \han}(\ZZ_p^d,M)$ and for every $1 \leq i \leq d$ we have $$\val_{\lambda}(\Delta_{p^l 1_i}(f)) \geq \val_{\lambda}(f) + c.$$
\end{lemma}

\begin{proof} Let $l\geq 0$ be arbitrary. We have the Vandermonde identity
$$\Delta_{p^l 1_i}(\binom{\un{x}}{\un{n}}) = \sum_{j=1}^{n_i}\binom{\un{x}}{\un{n}-j1_i}\binom{p^l}{j}.$$
Writing $f(\un{x}) = \sum_{\un{n}} a_{\un{n}} \binom{\un{x}}{\un{n}},$ the identity above shows that
$$\Delta_{p^l 1_i}(f) = \sum_{\un{n}} (\sum_{j=1}^\infty a_{\un{n} +j 1_i} \binom{p^l}{j}) \binom{\un{x}}{\un{n}},$$
and hence
$$\Delta^{\un{n}}(\Delta_{p^l 1_i}(f)(0)) = \sum_{j=1}^\infty a_{\un{n}+j1_i} \binom{p^l}{j}.$$

From this we get the lower bound
$$\val_M(\Delta^{\un{n}}(\Delta_{p^l 1_i}(f)(0))) \geq \inf_{j\geq 1}(\val(a_{\un{n}+j1_i}) + \val(\binom{p^l}{j})).$$

On the other hand, by definition we have the inequality
$$\val_M(a_{\un{n}+j1_i}) \geq \sum_{i=1}^d\lfloor  p^{\lambda}(n_i + j\delta_{ij})\rfloor + \val_{\lambda}(f).$$ Putting this all together, we get 
$$\val_{\lambda}(\Delta_{p^l 1_i}(f)) = \inf_{\un{n}}(\val_M(\Delta^{\un{n}}(\Delta_{p^l 1_i}(f)(0))) -\lfloor  p^{\lambda} \un{n} \rfloor) $$
$$ \geq \val_{\lambda}(f)+ \inf_{j\geq 1}(p^\lambda j -1 + \val(\binom{p^l}{j})),$$
and so any choice of $l$ which satisfies  $\inf_{j\geq 1}(p^\lambda j -1 + \val(\binom{p^l}{j})) > c$ works.
\end{proof}

\begin{corollary}
\label{cor:making_lambda_smaller}
For every $\lambda$ and $\lambda'$ there exists an $l \geq 0$ such that restriction from $\ZZ_p^d$ to $p^l \ZZ_p^d$ maps $\mathcal{C}^{\lambda\han}(\ZZ_p^d,M)$ into $\mathcal{C}^{\lambda'\han}(p^l \ZZ_p^d,M).$
\end{corollary}

\begin{proof}
 Choose any $c > p^{\lambda'} + 1$ and $l\geq 0$ so that the previous lemma applies to $c$ and $\lambda$. Let $f \in \mathcal{C}^{\lambda\han}(\ZZ_p^d,M)$ a function. We wish to show its restriction lies in  $\mathcal{C}^{\lambda'\han}(p^l\ZZ_p^d,M)$. We may assume $f$ lies in $\mathcal{C}^{\lambda\han}(\ZZ_p^d,M)^+.$ Let $\Delta_{p^l}^{\un{m}}$ denote the composition $\Delta^{m_1}_{p^l 1_1} \circ ... \circ \Delta^{m_d}_{p^l 1_d}.$ We need to show that $$\val(\Delta_{p^l}^{\un{m}}(f)(0)) - p^{\lambda'}|\un{m}| \rightarrow \infty.$$ 
 Applying the previous lemma successively to $f \in \mathcal{C}^{\lambda\han}(\ZZ_p^d,M)^+$ we get
the inequality
$$\val_\lambda(\Delta_{p^l}^{\un{m}}(f)) \geq |\un{m}|c.$$

By the general inequality $\val_M(g(0)) \geq \val_\lambda(g)$ we deduce
$$\val_M(\Delta_{p^l}^{\un{m}}(f)(0)) \geq  |\un{m}|c \geq  p^{\lambda'} |\un{m}| + |\un{m}|$$ which gives the desired 
 lower bound for the $\val_M(\Delta_{p^l}^{\un{m}}(f)(0))$.
\end{proof}

For $\lambda \leq \lambda'$ and $\mu\leq \mu'$ we have natural maps $\mathcal{C}^{\lambda'\text{-}\an}(\ZZ_p^d,M) \rightarrow \mathcal{C}^{\lambda\text{-}\an}(\ZZ_p^d,M)$ and $\mathcal{C}^{\an\text{-}\lambda',\mu'}(\ZZ_p^d,M) \rightarrow \mathcal{C}^{\an\text{-}\lambda,\mu}(\ZZ_p^d,M)$. We will often consider filtered colimits along these maps. The following lemma shows it does not matter much which system we use.

\begin{lemma}
\label{lem:analytic_cofinality}
The systems $\{\mathcal{C}^{\an\text{-}\lambda,\mu}(\ZZ_p^d,M)\}_{\lambda,\mu}$ and $\{\mathcal{C}^{\lambda\text{-}\an}(\ZZ_p^d,M)\}_\lambda$ are cofinal with respect to each other. Furthermore, up to cofinality the systems do not depend on the choice of basis of $\ZZ_p^d$.
\end{lemma}

\begin{proof}
The cofinality of $\{\mathcal{C}^{\an\text{-}\lambda,\mu}(\ZZ_p^d,M)\}_{\lambda,\mu}$ and $\{\mathcal{C}^{\lambda\text{-}\an}(\ZZ_p^d,M)\}_\lambda$ follows from the inequality $|\un{n}|_\infty \leq |\un{n}| \leq |\un{n}|_\infty/d$ for $\un{n} \in \ZZ^d_{\geq 0}$ by condition 1 of Proposition \ref{prop:equiv_conds_analytic_func}. The cofinality with respect to a different basis follows from Lemma \ref{lem:val_bound_gen_elem} and condition 2 of Proposition \ref{prop:equiv_conds_analytic_func}.
\end{proof}

\begin{definition}
\label{def:analytic_functions}
We let $\mathcal{C}^{\la}(\ZZ_p^d,M)$ be the colimit of any of these cofinal systems. Elements lying in this module are thought of as locally analytic functions of $\ZZ_p^d$ which are valued in $M$.
\end{definition}

\subsection{Locally analytic vectors}

Let $M$ be as in the previous subsection and let $G$ be a compact $p$-adic Lie group of dimension $d$ which acts continuously on $M$ by isometries. There exists an open uniform subgroup $G_0 \subset G$ by Corollary 4.3 of \cite{DdMS03}. We can choose it to be normal. As $G_0$ is uniform, there exists an ordered basis $g_1,...,g_d$ such that the coordinate map $c:(x_1,...,x_d)\mapsto g_1^{x_1}\cdot...\cdot g_d^{x_d}$ gives a homeomorphism of $\ZZ_p^d$ with $G_0$. Set $G_i = G^{p^i} = \{g^{p^i}: g \in G_0\}$. The $G_i$ are open normal subgroups of $G_0$ which correspond the $p^i \ZZ_p^d$ under $c$. We define $\mathcal{C}^0(G_0,M), \mathcal{C}^{\an\text{-}\lambda,\mu}(G_0,M), \mathcal{C}^{\lambda\text{-}\an}(G_0,M)$ and $\mathcal{C}^{\la}(G_0,M)$ by pulling back along $c$ the definitions of Definition \ref{def:analytic_funcs}. We are not claiming all of these are independent of the coordinate system chosen, but we shall see later the independence of the coordinate system for the most interesting objects considered here.

If $g\in G_0$ there exists some maximal $i$ so that $g^{p^{-i}} \in G_0$; there then exists some basis of $G_0$ with $g^{p^{-i}}=g_1$. For this basis, shifting in $G_0$ by $g$ corresponds to shifting in $\ZZ_p^d$ by $c(g)$. It then follows from Proposition \ref{prop:shifts_action} that $G_0$ acts on $\mathcal{C}^{\la}(G_0,M)$ by the formula $g(f)(x) = g(f(g^{-1}x))$.

\begin{definition} \label{ref:loc_analytic}
We define locally analytic elements as follows.

1. Set\footnote{Even though $G_0$ does not act on the submodules $\mathcal{C}^{\lambda\text{-}\an}(G_0,M)$ and $\mathcal{C}^{\an\text{-}\lambda,\mu}(G_0,M)$ in general, it does act on the larger space $\mathcal{C}^\la(G_0,M)$, so it makes sense to speak of $G_0$-fixed points.} $M^{G_0,\lambda\text{-}\an} = \mathcal{C}^{\lambda\text{-}\an}(G_0,M)^{G_0}$ and $M^{G_0,\an\text{-}\lambda,\mu} = \mathcal{C}^{\an\text{-}\lambda,\mu}(G_0,M)^{G_0}$.

2. Let $M^{G_0\hla} := \mathcal{C}^\la(G_0,M)^{G_0}$. We have $$M^{G_0\hla} = \varinjlim_{\lambda}M^{G_0,\lambda\text{-}\an} = \varinjlim_{\lambda,\mu}M^{G_0,\an\text{-}\lambda,\mu}.$$ 
\end{definition}

There is a natural injective map $M^{G_0\hla} \hookrightarrow M$ obtained by mapping $f:G_0\rightarrow M$ to $f(1)$. Thus we will interchangeably regard $M^{G_0\hla}$ as the submodule of $G_0$-locally analytic elements in $M$. Under this injection, $M^{G_0\hla}$ is actually stable under the entire $G$-action (not just the $G_0$-action). Indeed given $m \in M^{G_0\hla}$, $g\in G$ and $g_0 \in G_0$, we may write $g_0g(m)=g(g^{-1}g_0g)(m)$. Since $G_0$ is normal, $g_0 \mapsto g^{-1}g_0g$ is an automorphism of $G_0$, hence of $\ZZ_p^d$ if we choose for $c$ the Lie algebra coordinates; this shows that $g_0 \mapsto (g^{-1}g_0g)(m)$ is still an analytic function by Lemma \ref{lem:analytic_cofinality}. Hence the same holds for $g_0 \mapsto g(g^{-1}g_0g)(m)$ as $g$ acts on $M$ as an isometry.

\begin{lemma}
\label{lem:analyticity_group_functoriality}
If $G_0' \subset G_0$ are uniform subgroups of $G$ then $M^{G_0\text{-}\la} \subset M^{G_0'\text{-}\la}.$
\end{lemma}

\begin{proof}
Let $f \in \mathcal{C}^{\an\text{-}\lambda,\mu}(G_0,M)$. By Lemma \ref{lem:val_bound_gen_elem} and condition 2 of Proposition \ref{prop:equiv_conds_analytic_func} we have that $f$ restricts to $
\mathcal{C}^{\an\text{-}\lambda',\mu'}(G_0',M)$ for some $\lambda',\mu'$, from which the lemma follows. 
\end{proof}

\begin{definition}
\label{def:locally_analytic_vectors}
The locally analytic elements of $M$ are given by $$M^{\la} = \varinjlim_{G_0}{M^{G_0\hla}},$$
the colimit taken over uniform normal subgroups $G_0$ of $G$.
\end{definition}

\begin{remark}
We have made the definition above to get a concept of locally analytic vectors independent of the choice of $G_0$, which will suffice for this article. Of course, one already expects $M^{G_0\hla}$ to be independent of the choice of $G_0$, as is known in the characteristic 0  theory. We plan to address this and other foundational issues in future work.
\end{remark}

\begin{example}
\label{example:loc_analytic_elems}

1. Suppose $R= \QQ_p$ so that $M$ is a $\QQ_p$-vector space with $\val_M$ making it a Banach space. In this case, $M^{\la}$ conicides with the usual locally analytic vectors (as defined for example in \cite{Eme17}). This is a consequence of the Amice theorem (Chapitre 3 of \cite{Am64}, see also Théorème I.4.7 of \cite{Co10}).

In this setting, the $G_0$-action can be differentiated and one gets an action of the Lie algebra $\mathrm{Lie}(G)$, which is a key extra structure. Unfortunately, this action seems to be unavailable outside this special case. The reason is that when $p$ is the topologically nilpotent unit, one can link functions of the form $\binom{x}{n}$ and $x^n$ by cancelling out the $p$-part of the denominator of $\binom{x}{n}$ (see the proof of Théorème I.4.7 of \cite{Co10}). But it unclear (at least to us) how to do something of this flavour more generally.

2. Suppose $M$ is an $\FF_p$-vector space. In this case, one checks using Proposition \ref{prop:equiv_conds_analytic_func} that $M^{\la}$ coincides with the super-Hölder vectors $M^{\sh} \subset M$ defined in \cite{BR24}. 

3. Suppose we have a perfectoid field $\widehat{K}_\infty$ in characteristic 0 with an action of $\Gamma \cong \ZZ_p$. Recall that we can form the tilt $\widehat{K}_\infty^{\flat}$ which is a perfectoid field in characteristic $p$, and there is a multiplicative map denoted $\sharp$ from $\widehat{K}_\infty^{\flat}$ to $\widehat{K}_\infty$. For $x \in \widehat{K}_\infty^{\flat,+}$ it is given by $x^\sharp = \lim_{n \rightarrow \infty} y_n^{p^n}$ where the $y_n$ are any lifts of $x^{1/p^n}$ to $\widehat{K}_\infty^+.$



\textbf{Claim.} Suppose that $x^\sharp$ is $\Gamma$-smooth. Then $x$ is locally analytic. 

Indeed, without loss of generality, we may assume $x^\sharp$ is fixed by $\gamma$ and $x\in \widehat{K}_\infty^{\flat,+}.$ 
The element $\gamma^p$ fixes $x^{1/p}$, because the action on the $p$'th roots of $x$ gives a homomorphism $\ZZ_p \rightarrow S_n$ with $n \leq p$, and any such homomorphism has to factor through $\ZZ/p.$ Arguing similarly by induction, we see that  $(x^{1/p^m})^\sharp$ is fixed by $\gamma^{p^m}$.

Hence, we get
$$\val_{\widehat{K}_\infty^\flat}(\gamma^{p^m}((x^{1/p^m}))-(x^{1/p^m})) = 
\val_{\widehat{K}_\infty}([\gamma^{p^m}((x^{1/p^m}))-(x^{1/p^m})]^\sharp)$$

which is $$\geq \min(\val(p),\val([\gamma^{p^m}((x^{1/p^m})^\sharp)-(x^{1/p^m})^\sharp]))=\val(p).$$

Hence $\val(\gamma^{p^m}(x)-x) \geq p^m \val(p)$ which shows that $x$ is locally analytic. This concludes the proof of the claim.

For example, take the cyclotomic extension for $\widehat{K}_\infty$ with $\Gamma$ being the cyclotomic group. The tilt $\widetilde{\mathbf{E}}_{\QQ_p}$ is isomorphic to the $X$-adic completion of $\cup_n \FF_p((X^{1/p^n}))$. The element $X \in \widetilde{\mathbf{E}}_{\QQ_p}$ has $X^\sharp = \zeta_p - 1 \in K_\infty$, and so is locally analytic. This is of course also easy to check directly. Let us point out that the main theorem of \cite{BR22} gives the much stronger $$\widetilde{\mathbf{E}}^\la_{\QQ_p} = \cup_n \FF_p((X^{1/p^n}))$$ ($\widetilde{\mathbf{E}}^\sh$ in their notation) which shows that taking locally analytic vectors undoes the $X$-adic completion. In $\mathsection{4}$ we shall show how to deduce this result from our main theorem.
\end{example}

In Example \ref{example:loc_analytic_elems}.3 we had a field whose locally analytic vectors were also a field. This happens in general, as shown in the following.
\begin{lemma}
(i). The abelian subgroup $R^\la \subset R$ is a subring.

(ii). Let $r \in R^\times \cap R^\la$. Then $r^{-1}\in R^\la.$
\end{lemma}

\begin{proof}
This is similar to argument of Lemma 2.5 of \cite{BC16} which we produce here for the convenience of the reader. 
We may assume $G=G_0$ is a uniform subgroup and identify it analytically with $\ZZ_p^d$. Recall that locally analytic functions are those for which the valuation of $a_{\un{n}}$ grows at least linearly with $|\un{n}|.$
For (i), note that for $\un{n},\un{m}\in \ZZ^d_{\geq 0}$ we have
\begin{align}
\label{identity:binomial_multiplication}
\binom{\un{x}}{\un{n}}\cdot\binom{\un{x}}{\un{m}}=\sum_{\un{k}\leq\un{n}+\un{m}}c_{\un{k}}\binom{\un{x}}{\un{k}} 
\end{align}
for some $c_{\un{k}} \in \ZZ.$ This proves $R^\la$ is a ring. 
For (ii), write $\orb_r(\un{x})= r+\sum_{\un{n}
\neq 0}a_{\un{n}}\binom{\un{x}}{\un{n}}$. We have the identity
$$\orb_{1/r}(\un{x})=\frac{1}{r}\sum_{j \geq 0}(-1)^j(\sum_{\un{n}\neq 0}a_{\un{n}}\binom{\un{x}}{\un{n}}/r)^j,$$
which is still locally analytic.
\end{proof}

\begin{proposition}
\label{prop:locally_analyic_basis}
Suppose that $M$ is free of rank $d$ over $R$, and suppose $m_1,...,m_d$ is an $R$-basis of $M$ with $g \mapsto \mathrm{Mat}(g)$ locally analytic in each coordinate. Then $M^\la = \oplus_i R^\la \cdot m_i.$ 
\end{proposition}

\begin{proof}
Again, we reproduce the argument of Proposition 2.3 of \cite{BC16} for the convenience of the reader. Let $a_{ij}(g)$ be the coordinates of $g\mapsto\mathrm{Mat}(g)$. Then $g(m_i)=\sum_{j}a_{ij}(g)m_j,$ so each $m_i$ is locally analytic. This proves the inclusion $\oplus_i R^\la \cdot m_i \subset M^\la$. Conversely, let $m \in M^\la$. Write $m=\sum_i r_i m_i$ with $r_i \in R.$ By assumption, we may write $g(m) = \sum_i f_i(g)m_i$ with each $f_i$ locally analytic. Applying $g$ to $m = \sum_i r_i m_i$, we get the identity $g(r_i) = \sum_{j}b_{ij}(g)f_j(g),$ where $b_{ij}(g)$ are the coordinates of $\mathrm{Mat}(g)^{-t}$. These are polynomials in the $a_{ij}(g)$ multiplied by $\det(g)$, hence are locally analytic by the previous lemma.
\end{proof}

\subsection{Higher locally analytic elements}

It will be useful for us to define a derived functor for $M \mapsto M^{\la}$. 
As $M^{\la} = \varinjlim_{G_0}\mathcal{C}^{\la}(G_0,M)^{G_0}$, we may extend this definition to $i\geq 0$ by setting
$$\R^{i}_{\la}(M) := \varinjlim_{G_0}\mathrm{H}^{i}(G_0,\mathcal{C}^{\la}(G_0,M)).$$

Here we are considering continuous cocycles, taking the inductive topology on $\mathcal{C}^{\la}(G_0,M)$ induced from that of its submodules $\mathcal{C}^{\lambda\text{-}\an}(G_0,M).$ 

We shall call these groups the higher locally analytic elements of $M$. 
If 
$$0 \rightarrow M_1 \rightarrow M_2 \rightarrow M_3 \rightarrow 0$$
is a short exact sequence of submodules in the appropriate category then we claim we have a long exact sequence
$$0 \rightarrow M_1^{\la} \rightarrow M_2^{\la} \rightarrow M_3^{\la} \rightarrow \R^1_{\la}(M_1) \rightarrow \R^1_{\la}(M_2) \rightarrow \R^1_{\la}(M_3) \rightarrow ...$$

This requires some explanation. Using the open mapping theorem in this setting 
(Theorem 2.2.8 of \cite{KL11}) we conclude all the mappings are strict. From here the argument for exactness is the same as Lemma 2.2.2 of \cite{Pan22}.

\section{Decompletion}
In this section we shall prove decompletion results for $G$-modules $M$ as in $\mathsection{2}$ under more assumptions. Subsection $\mathsection{3.1}$ will state the results while the rest of the subsections will be devoted to their proof.

\subsection{Statement of the results}

The set up is as follows. Let $(\WLambda, \WLambda^+)$ be a pair of topological rings with $\WLambda^+ \subset \WLambda$ and $\varpi$ an element of $\WLambda^+$. We assume $\WLambda$ is a Tate algebra over $\ZZ_p$ with $\WLambda^+$ a ring of definition, and $\varpi$ a topologically nilpotent unit. We assume $\WLambda^+$ is $\varpi$-adically complete, and endow $\WLambda$ with a valuation $\val_{\WLambda}$ making it so that $\WLambda^+=\WLambda^{\val\geq0}$ and $\val_{\WLambda}(\varpi x) = \val_{\WLambda}(\varpi)+ \val_{\WLambda}(x)$.
Finally, we assume we are given a compact $p$-adic Lie group $G$ acting on $\WLambda$ by isometries. We let $M$ be a finite free module over $\WLambda$, endowed with a $\varpi$-adic topology and a semilinear $G$-action. We may choose a $G$-stable $\WLambda^+$-lattice $M^+ \subset M$ so that the topology on $M$ is induced from the $\varpi$-adic valuation $\val_M$ making $M^+$ the open unit ball (this valuation is implicitly needed to make sense of $\R^i_{\la}(M)$ in what follows).

The main result of this section is the following result which generalizes Theorem C of \cite{Po24}.

\begin{theorem}
\label{main_thm}
Suppose that the pair $(\WLambda, \WLambda^+)$ satisfies the Tate-Sen axioms (TS1)-(TS4) (see $\mathsection{3.2}$). 
Then

1. The natural map $$\WLambda \otimes_{\WLambda^{\la}} M^{\la} \rightarrow M$$ is an isomorphism.

2. If moreover the Lie algebra $\mathrm{Lie}(G)$ is abelian, 
we have $\R_{\la}^i(M) = 0$ for $i \geq 1$.

\end{theorem}

\begin{corollary}
\label{cor:coh_comparison_dim_1}
Let $M$ be as above, and assume further that $G\cong\ZZ_p$. Then there exist natural isomorphisms $$\mathrm{H}^{i}(G,M^{\la}) \cong \mathrm{H}^{i}(G,M)$$ for $ i\geq 0$.
\end{corollary}

\begin{proof}
As $\Gamma \cong \ZZ_p$ we know that $\mathrm{H}^{i}(G,M) = \mathrm{H}^{i}(G,M^{\la}) = 0$ for $i \geq 2$ (see V, 2.2.3.3 of \cite{La65}), while the equality for $i=0$ is obvious. So the only nontrivial case is $i=1$. Choosing $\gamma$ for a generator of $G$, we have an exact sequence
$$0\rightarrow M/\mathrm{H}^{0}(G,M) \xrightarrow{\gamma-1} M \rightarrow \mathrm{H}^{1}(G,M)\rightarrow 0.$$

Taking locally analytic vectors (regarding $\mathrm{H}^{0}(G,M)$ and $\mathrm{H}^{1}(G,M)$ as having a trivial $G$-action) we obtain an exact sequence
$$0\rightarrow (M/\mathrm{H}^{0}(G,M))^{\la} \xrightarrow{\gamma-1} M^{\la} \rightarrow \mathrm{H}^{1}(G,M)\rightarrow \R_{\la}^1(M/\mathrm{H}^{0}(G,M)).$$
First, we claim $\R_{\la}^1(M/\mathrm{H}^{0}(G,M))=0$. This is because there is an exact sequence $$\R_{\la}^1(M) \rightarrow \R_{\la}^1(M/\mathrm{H}^{0}(G,M)) \rightarrow \R_{\la}^2(\mathrm{H}^{0}(G,M))$$ and both of its outer terms vanish (the $\R_{\la}^1$ term because of the theorem, the $\R_{\la}^2$ term because $\Gamma \cong \ZZ_p$). 

On the other hand, we claim that $(M/\mathrm{H}^{0}(G,M))^\la = M^{\la}/\mathrm{H}^{0}(G,M).$ This reduces to showing $\R^1_{\la}(\mathrm{H}^{0}(G,M)) = 0$; for this vanishing we need to show that $\gamma-1$ acts surjectively on $\mathcal{C}^{\la}(\Gamma,\mathrm{H}^{0}(G,M))$. Now if $\sum_n a_n \binom{x}{n} \in \mathcal{C}^{\la}(\Gamma,\mathrm{H}^{0}(G,M))$  with the coordinate $x$ corresponding to $\gamma$, then $(\gamma-1)(\sum_n a_n \binom{x}{n}) = \sum_n a_n \binom{x}{n-1}$, from which the surjectivity is clearly seen.

Putting the two claims together, we obtain an exact sequence
$$ 0\rightarrow M^{\la}/\mathrm{H}^{0}(G,M) \xrightarrow{\gamma-1} M^{\la} \rightarrow \mathrm{H}^{1}(G,M)\rightarrow 0$$

which shows $\mathrm{H}^{0}(G,M)$ is the cokernel of $M^{\la}\xrightarrow{\gamma-1} M^{\la}$, that is to say naturally isomorphic to $\mathrm{H}^{1}(G,M^{\la})$, as required.
\end{proof}

We end this subsection with a few natural conjectures. 

\begin{conjecture}
\label{conj:decompletion} The statement of Theorem \ref{main_thm} holds for any compact p-adic Lie group $G$ (with no assumption on $\mathrm{Lie}(G)$).
\end{conjecture}

The following is an analogy of Theorem 1.5 of \cite{RJRC22}. 

\begin{conjecture}
\label{conj:spectral_seq} There exists a spectral sequence $$E_2^{i,j}=\mathrm{H}^i(G,\R_{\la}^j(M)) \Rightarrow \mathrm{H}^{i+j}(G,M).$$  
\end{conjecture}

Finally, we have a conjecture regarding cohomology which follows from the previous two conjectures.

\begin{conjecture}
\label{conj:cohomology} Suppose $\WLambda$ satisfies the Tate-Sen axioms (TS1)-(TS4) (see $\mathsection{3.2}$ below), and that $G$ is a compact $p$-adic Lie group. Then there exist natural isomorphisms $$\mathrm{H}^{i}(G,M^{\la}) \cong \mathrm{H}^{i}(G,M)$$ for $ i\geq 0$. 
\end{conjecture}

\subsection{The Tate-Sen method.}

In this section we recall the Tate-Sen method introduced in \cite{BC08} as well as some extensions of it introduced in \cite{Po22b} and \cite{Po24}.

Let $\WLambda$ and $G$ be as in $\mathsection{3.1}$. We let $G_0$ be an open and normal uniform subgroup of $G$ (such a subgroup always exists by Corollary 4.3 of \cite{DdMS03}). We suppose $G_0$ is endowed with a continuous character $\chi : G_0 \rightarrow \ZZ_p^\times$ with open image and let $H_0 = \ker \chi$.
If $g \in G_0$, let $n(g) = \val_p(\chi(g)-1) \in \ZZ$. For $G_0'$ an open
subgroup of $G_0$, set $H' = G_0' \cap H_0$. Let $G_{H'}$ be the normalizer of $H'$ in $G_0$.
Note $G_{H'}$ is open in $G_0$ since $G'_0 \subset G_{H'}$. Finally let $\widetilde{\Gamma}_{H'} = G_{H'}/H'$ and write
$C_{H'}$ for the center of $\widetilde{\Gamma}_{H'}$. By Lemma 3.1.1 of \cite{BC08}  the group $C_{H'}$ is open in
$\widetilde{\Gamma}_{H'}$. Let $n_1(H')$ be the smallest positive integer such that $\chi(C_{H'})$ contains $1+p^n\ZZ_p$.

The Tate-Sen axioms are the following:

\textbf{(TS1)} There exists $c_1 > 0$ such that for each pair $H_1 \subset H_2$ of open subgroups
of $H_0$ there exists $\alpha \in \WLambda^{H_1}$ such that $\val(\alpha) > -c_1$ and $\sum_{\tau \in H_2/H_1} \tau(\alpha) = 1$.

\textbf{(TS2)} There exists $c_2 > 0$ and for each open subgroup $H$ of $H_0$ an integer
$n(H)$, as well as an increasing sequence $(\Lambda_{H,n})_ {n\geq n(H)}$ of closed subalgebras
of $\WLambda^{H}$, each containing $\varpi^{\pm1}$, and $\Lambda_{H,n}$-linear maps $R_{H,n} : \WLambda^{H} \rightarrow \Lambda_{H,n}$
such that

(1) If $H_1 \subset H_2$ then $\Lambda_{H_2,n} \subset \Lambda_{H_1,n}$ and $R_{H_1,n}|_{\WLambda^{H_2}} = R_{H_2,n}$.

(2) $R_{H,n}(x) = x$ if $x \in \Lambda_{H,n}$.

(3) $g(\Lambda_{H,n}) = \Lambda_{gHg^{-1},n}$ and 
$g(R_{H,n}(x)) = R_{gHg^{-1}}(gx)$ if $g \in G_0$.

(4) If $n \geq n(H)$ and if $x \in \WLambda^{H}$ then $\val(R_{H,n}(x)) \geq \val(x) - c_2$.

(5) If $x \in \WLambda^{H}$ then $\lim_{n\rightarrow \infty}(R_{H,n}(x)).$

\textbf{(TS3)} There exists $c_3 > 0$ and, for each open subgroup $G'$ of $G_0$ an integer
$n(G') \geq n_1(H')$ where $H' = G' \cap H_0$, such that if $n(\gamma) \leq n \leq n(G')$ for $\gamma \in \widetilde{\Gamma}_{H'}$ then $\gamma-1$ is invertible on $X_{H',n} = (1 - R_{H',n})(\WLambda^{H'})$ and $\val((\gamma - 1)^{-1}(x)) \geq \val(x)-c_3.$

\textbf{(TS4)} For any sufficiently small open $G' \subset G_0$  and $n \geq n(G')$, there exists a positive real number $t = t(G', n) > 0$ such that if $\gamma \in G'$ and $x \in \Lambda_{H',n}$ then $\val((\gamma - 1)(x)) \geq \val(x) + t$. 

Note that with the exception of (TS4), this is exactly the setting of $\mathsection{4.1}$ of \cite{Po22b}, except that we do not assume $\chi$ is defined on the entire group. Also, we have slightly switched notation: our $H',G'$ are denoted there by $H,G$. We apologize for the possible confusion. We now recall some definitions and results from $\mathsection{4.4}$ of loc. cit. which are in turn a small variant of the content appearing in $\mathsection{3}$ of \cite{BC08}.

We have the following two additive tensor categories, where $G'$ is an open subgroup of $G_0$:

(1). $\Mod^{G_0}_{\WLambda^+}(G')$, the category of 
finite free $\WLambda^+$-semilinear representations
of $G_0$ such that for some basis $\val(\Mat(\gamma)-1) > c_1+2c_2+2c_3$ for $\gamma \in G'$.

(2). $\Mod^{G_0}_{\Lambda_{H',n}^+}(G')$, the category of finite free $\Lambda_{H',n}^+$-semilinear representations of $G_0$ that are fixed by $H' = G'\cap H_0$ and which have a $c_3$-fixed basis for the $G'$ action, i.e. for some basis $\val(\mathrm{Mat}(\gamma)-I)>c_3$ for $\gamma \in G'.$

\begin{definition}
Let $M^+ \in \Mod^{G_0}_{\WLambda^+}(G')$. We let $D_{H',n}^+(M^+)$ be the union of all finitely generated $\Lambda^+_{H',n}$-submodules
of $M^+$ which are stable by $G_0$, fixed by $H'$ and which are generated by a $c_3$-fixed set of generators. We let $D_{H',n}(M) = D_{H',n}^+(M^+)[1/\varpi]$.
\end{definition}

The following is Proposition 4.9 of \cite{Po22b}. It uses only the axioms (TS1)-(TS3).

\begin{proposition}
\label{prop:dn_decompletion}
Let $n \geq n(G').$
 The association $M^+\mapsto D_{H',n}^+(M^+)$ gives an equivalence of categories from $\Mod^{G_0}_{\Lambda_{H',n}^+}(G')$ to $\Mod^{G_0}_{\Lambda_{H',n}^+}(G')$. The inverse functor is given by $D^+\mapsto \WLambda^+\otimes_{\Lambda_{H',n}^+}D^+.$
\end{proposition}

A few words on this proposition are in order. This kind of Tate-Sen decompletion result has a long history going back to Tate's paper on $p$-divisible groups (see $\mathsection{3}$ of \cite{Ta67} for the first ancestor of these ideas). Roughly speaking, this equivalence should be viewed as a form of descent from representations where the action of $\gamma-1$ is sufficiently contracting to analytic representations. It has also been understood to have connection with $p$-adic Simpson theory, see for example Theorem 3.4 of \cite{Wa23}. We should note that when $p$ is topologically nilpotent in $\WLambda$, i.e. in characteristic zero, it appears implicitly in $\mathsection{3}$ of \cite{BC08}, see also $\mathsection{5}$C of \cite{Po24} where it was spelled out. In \cite{Po22b} it was generalized to allow for $\ZZ_p$-Tate algebra coefficients.

\begin{lemma}
\label{lem:D_n_tensor_prod}
Let $N^+, M^+ \in \Mod^{G_0}_{\WLambda^+}(G')$. Then $N^+ \otimes_{\WLambda^+} M^+$ lies in $\Mod^{G_0}_{\WLambda^+}(G')$ and $D^+_{H',n}(M^+) \otimes_{\Lambda^+_{H',n}} D^+_{H',n}(M^+) = D^+_{H',n}(M^+ \otimes_{\WLambda^+} N^+)$.
\end{lemma}

\begin{proof}
Suppose $\{e_i\}_{1\leq i \leq n}$ and $\{f_j\}_{1\leq j \leq m}$ are respective bases of $N^+$ and $M^+$ such that $\val(\mathrm{Mat}_{\{e_i\}}(\gamma)-1),\val(\mathrm{Mat}_{\{f_j\}}(\gamma)-1) > c_1+2c_2+2c_3$ for $\gamma \in G'.$ Using the identity
$$(\gamma-1)(e_i\otimes f_j)-e_i\otimes f_j=(\gamma-1)(e_i)\otimes\gamma(f_j)+e_i\otimes (\gamma-1)(f_j),$$
we see that $\{e_i \otimes f_j\}_{1\leq i \leq n, 1\leq j \leq m}$ is a basis of $N^+ \otimes_{\WLambda^+} M^+$ with $$\val(\mathrm{Mat}(\gamma)-1) > c_1+2c_2+2c_3.$$ Hence, $N^+ \otimes_{\WLambda^+} M^+$ lies in $\Mod^{G_0}_{\WLambda^+}(G')$. Next, we show that 
$$D^+_{H',n}(M^+) \otimes_{\Lambda^+_{H',n}} D^+_{H',n}(M^+) =D^+_{H',n}(M^+ \otimes_{\WLambda^+} N^+).$$
Both the left and right hand side are objects of $\Mod^{G_0}_{\Lambda_{H',n}^+}(G')$, so by virtue of Proposition \ref{prop:dn_decompletion} it is enough to show they are equal after tensoring with $\WLambda^+$. To show this, we compute
\begin{align*}
\WLambda^+\otimes_{\Lambda^+_{H',n}}(D^+_{H',n}(M^+) \otimes_{\Lambda^+_{H',n}} D^+_{H',n}(M^+)) &= M^+ \otimes_{\Lambda^+_{H',n}} D^+_{H',n}(M^+) \\ 
&=M^+ \otimes_{\WLambda^+}(\WLambda^+\otimes_{\Lambda^+_{H',n}} D^+_{H',n}(M^+))\\
&= M^+ \otimes_{\WLambda^+}N^+\\
&= \WLambda^+\otimes_{\Lambda^+_{H',n}}D^+_{H',n}(M^+ \otimes_{\WLambda^+} N^+).
\end{align*}
This completes the proof. \end{proof}


We have the following result which allows us to control the action of $\gamma-1$ on $D_{H',n}(M)$.

\begin{proposition}
\label{prop:action_uniformity} 
If (TS4) holds then there exists some element $s = s(G',n)> 0$, independent of $M$, so that for $\gamma \in G'$ and $x \in D_{H',n}(M)$ we have $\val((\gamma-1)(x)) \geq \val(x) + s$.
\end{proposition}
\begin{proof}
 We know that $D_{H',n}^+(M)^+$ has a basis of $c_3$-fixed elements. Using the identity 
$$(\gamma-1)(ab)=(\gamma-1)(a)b + \gamma(a)(\gamma-1)(b)$$
we see that $(\gamma-1)(D_{H',n}^+(M)^+) \subset \varpi^{s}D_{H',n}^+(M)^+$ for $s = \mathrm{min}(c_3,t)$ where $t = t(G',n)$ is as in (TS4). If $x \in \varpi^n\cdot D_{H',n}^+(M)^+$ then writing $x = \varpi^n y$ and using the same identity as before we conclude that $(\gamma-1)(x) \in \varpi^{n+s}\cdot D_{H',n}^+(M)^+.$ This concludes the proof.
\end{proof}

\begin{corollary}
\label{cor:descent_loc_analytic}
If (TS4) holds then $D_{H',n}(M) \subset M^{G'\hla}.$
\end{corollary}
\begin{proof}
Use the the expansion $\gamma^{x}(m)=\sum_{n\geq0}(\gamma-1)^{n}(m)\binom{x}{n}$ for $x\in \ZZ_p$.
\end{proof}

We are now in a position to prove the first part of our main theorem.

\textsl{Proof of part 1 of Theorem \ref{main_thm}}. For $G'$ small enough and $n$ large enough the descent to $D_{H',n}$ applies (this can always be achieved because of continuity). We need to show that the natural map
$$\WLambda \otimes_{{\Lambda}^{\la}} M^{\la} \rightarrow M$$
is an isomorphism. We shall show this map is an isomorphism for $G'$ small enough when $M^{\la}$ is replaced with $M^{G'\hla}$, which suffices (the transition maps along different $G'$'s are then forced to be isomorphisms). By Corollary \ref{cor:descent_loc_analytic} we have $D_{H',n} \subset M^{G'\hla}$ and we know $M$ descends to $D_{H',n}$, so the $\Lambda_{H',n}$-basis of $D_{H',n}$ gives a $\WLambda$-basis of $M$ consisting of $G'$-locally analytic elements. By Proposition \ref{prop:locally_analyic_basis} we have $M^{G' \hla} = \WLambda^{G'\hla} \otimes_{\Lambda_{H',n}} D_{H',n}$ and so the natural isomorphism $\WLambda \otimes_{\Lambda_{H',n}} D_{H',n} \xrightarrow{\cong} M$ is identified with the map $\WLambda \otimes_{{\Lambda}^{G'\hla}} M^{G'\hla} \rightarrow M$, which concludes the proof.

\subsection{Descending analytic functions}

From now on we suppose (TS1)-(TS4) hold and that $\mathrm{Lie}(G)$ is abelian. It follows that $G_0$ is abelian also. By Proposition \ref{prop:shifts_action} we have that $\mathcal{C}^{\lambda\han}(G_0,M)$ is a $G_0$-module for every $\lambda \in \RR$. Recall (Definition \ref{def:analytic_funcs}) that $\mathcal{C}^{\lambda\han}(G_0,M)$ is endowed with the valuation
$$\val_{\lambda}(f) = \inf_{\un{n}}(\val(a_{\un{n}}(f))- \lfloor p^\lambda \un{n} \rfloor).$$

We set
$$M_{\un{n},\lambda} = \{f \in \mathcal{C}^{\lambda\han}(G_0,M): \Delta^{\un{n}}(f) = 0\}$$ with its induced valuation. We now note a few properties of the $\widetilde{\Lambda}$-modules $M_{\un{n},\lambda}$. First, each $M_{\un{n},\lambda}$ is preserved under the $G_0$-action because $\Delta^{\un{n}}$ is $G_0$-equivariant, and so is actually a $G_0$-semilinear submodule of $\mathcal{C}^{\lambda\han}(G_0,M)$. We have $M_{\un{n},\lambda} \subset M_{\un{n'},\lambda}$ when $\un{n} \leq \un{n'}$. The limit
$\varinjlim_{\un{n}}{M_{\un{n},\lambda}} \subset \mathcal{C}^{\lambda\han}(G_0,M)$, indexed over $\un{n} 
\in \ZZ_{\geq 0}^d$, is dense in $\mathcal{C}^{\lambda\han}(G_0,M)$. Let $m_1, ..., m_{\mathrm{rank}(M)}$ be a $\WLambda^+$ basis of $M^+$, which we choose once and for all. One sees that $M_{\un{n},\lambda}$ is finite free over $\widetilde{\Lambda}$ with a basis given by the functions $\un{x}\mapsto m_i\varpi^{  \lfloor p^{\lambda}\un{k}\rfloor }\binom{\un{x}}{\un{k}}$ for $\un{k}\leq \un{n}$ and $1\leq i \leq \mathrm{rank}(M)$. Finally we note this implies $M_{\un{n},\lambda}\cdot M_{\un{n'},\lambda} \subset M_{\un{n}+\un{n'},\lambda}$ when $M=\WLambda$.

\begin{definition}
\label{def:c_small_pair}

1. A subgroup $G_0$ is called $c$-small if $\val(g-1)(\varpi) > c$ and if $\val(g-1)(m_i) \geq c$ for each basis element $m_i$. 

2. An element $\lambda \in \RR$ is called $c$-small if $\lambda > \log_p(c+1).$

3. We say a pair $(G_0,\lambda)$ is $c$-small if both of $G_0$ and $\lambda$ are $c$-small.

\end{definition}

\begin{warning}
\label{warning: c_smalness}
The functoriality of $c$-smallness does not coincide with that of $(G_0, \lambda) \mapsto \mathcal{C}^{\lambda \han}(G_0,M)$. Namely, if the pair $(G_0, \lambda)$ maps to the pair $(G_0', \lambda')$ in the direct limit defining the $\R^{i}_{\la}$ then it is not true in general that $(G_0', \lambda')$ is still $c$-small. This is because $(G_0, \lambda)$ maps to $(G_0', \lambda')$ when $G_0' \subset G_0$ and $\lambda' \leq \lambda$, but the $c$-smallness of $(G_0', \lambda')$ is only guaranteed when $\lambda' \geq \lambda$.
\end{warning}

\begin{lemma}
\label{lem:small_action}
Let $c > 0$, $\lambda \in \RR$ and $G_0$ a subgroup.

1. If $G_0$ is $c$-small then there exists some $l \geq 0$ depending on $\lambda$ and $c$ such that for every $g \in G_l$ and for every $\un{n}$ one has $\val(\Mat(g)-I) \geq c$ for $M_{\un{n},\lambda}$ (taking the $\WLambda$-basis described above).

2. If $(G_0, \lambda)$ is $c$-small, one can take $l=0$ in 1.
\end{lemma}

\begin{proof}
Each basis element of each $M_{\un{n},\lambda}$ is of the form $m_i\varpi^{\lfloor p^{\lambda}\un{k} \rfloor}\binom{\un{x}}{\un{k}}$ for some $m_i$ for $1 \leq i \leq \mathrm{rank}(M)$ and $\un{k} \leq {\un{n}}.$ Using the identity
$$(g-1)(abc) = (g-1)(a)bc + a(g-1)(b)g(c) + g(a)g(b)(g-1)(c),$$
one reduces to showing each of the valuations of $(g-1)(m_i)$, $(g-1)(\varpi^{\lfloor p^{\lambda}\un{k} \rfloor})$ and $(g-1)(\binom{\un{x}}{\un{k}})$ are $ > c$.

First, the valuation of the $(g-1)(m_i)$ is $> c$ by the $c$-smallness of $G_0$. Second, the valuation of the term $(g-1)(\varpi^{\lfloor p^{\lambda}\un{k} \rfloor})$ is $\geq \val((g-1)(\varpi)) > c$ by virtue of the identity $(g-1)(\varpi^n) = (g-1)(\varpi^{n-1})g(\varpi)+(g-1)(\varpi)\varpi.$ 

Finally, we need to estimate $(g-1)(\binom{\un{x}}{\un{k}})$. First, by using the identity $$(gh-1)(a) = g((h-1)(a)) + (g-1)(a),$$ we may reduce to the case where $g$ is a basis element, and hence its action is given by $\sh_{p^l1_i}$ for some $1 \leq i \leq d$.
We then have the Vandermonde identity 
$$(\sh_{p^l 1_i}-1)(\binom{\un{x}}{\un{k}}) = \sum_{j=1}^{k_i}\binom{\un{x}}{\un{k}-j1_i}\binom{p^l}{j}.$$

We have $\val_{
\lambda}(\binom{\un{x}}{\un{k}}) = -\lfloor p^\lambda\un{k} \rfloor$ and $\val_{
\lambda}(\binom{\un{x}}{\un{k}-j1_i} = -\lfloor p^\lambda(\un{k}-j1_i) \rfloor$, so it suffices to show
$$-\lfloor p^\lambda(\un{k}-j1_i) \rfloor + \val(\binom{p^l}{j}) >  -\lfloor p^\lambda\un{k} \rfloor+c$$ for every $j \leq k_i$. This inequality follows in turn from the simplier inequality
$$\lfloor p^\lambda j \rfloor + \val(\binom{p^l}{j}) > c.$$

If $(G_0,\lambda)$ is $c$-small then $p^\lambda > c + 1$ and the inequality holds for all $j \geq 1$ and $l=0$. Otherwise, it is clear one can choose $l$ large enough so that the inequality holds for all $j \geq 1$.
\end{proof}

Now choose any $c > c_1 + 2c_2 + 2c_3$, and suppose $(G_0,\lambda)$ is $c$-small. From the lemma,  each $M_{\un{i},\lambda}^+$ lies in the category  $\Mod^{G_0}_{\WLambda^+}(G_0)$. By the Tate-Sen method we have $\Lambda_{H_0,n}^+$-modules $D_{H_0,n}^+(M_{\un{i},\lambda}^+)$ lying in $\Mod^{G_0}_{\Lambda_{H_0,n}^+}(G_0)$ for $n \geq n_0(M,\lambda)$. The $D_{H_0,n}^+(M_{\un{i},\lambda}^+)$ form a direct system, and we set
\begin{align}
\label{def:big_Dplus}
D^{\lambda,+}_{H_0,n}(M^+) = \varinjlim_{\un{i}} D_{H_0,n}^+(M_{\un{i},\lambda}^+)
\end{align}
and $D^{\lambda}_{H_0,n}(M) = D^{\lambda,+}_{H_0,n}(M^+)[1/\varpi].$

The natural isomorphisms 
$$\WLambda^+ \otimes_{\Lambda_{H_0,n}^+} D_{H_0,n}^+(M_{\un{i},\lambda}^+) \xrightarrow{\cong} M_{\un{i},
\lambda}^+$$
glue to an isomorphism
\begin{align}
\label{iso:descent_of_funcs}
\WLambda^+ \otimes_{\Lambda_{H_0,n}^+} D^{\lambda,+}_{H_0,n}(M^+) \xrightarrow{\cong} \mathcal{C}^{\lambda\han}(G_0,M)^+
\end{align}

exhibiting $\mathcal{C}^{\lambda\han}(G_0,M)^+$ as a descent of $D^{\lambda,+}_{H_0,n}(M^+)$.

If we only assume that $G_0$ is $c$-small (but no assumptions on $\lambda$) then we still get a descent, but now each $M_{\un{i},\lambda}^+$ only lies in $\Mod^{G_l}_{\WLambda^+}(G_0)$ for some $l \geq 0$. In this case, we only get a descent to a module $D_{H_l,n}^+(M_{\un{i},\lambda}^+)$ lying in $\Mod^{G_l}_{\Lambda_{H_l,n}^+}(G_0)$ and an isomorphism
$$\WLambda^+ \otimes_{\Lambda_{H_l,n}^+} D^{\lambda,+}_{H_l,n}(M^+) \xrightarrow{\cong} \mathcal{C}^{\lambda\han}(G_0,M)^+.$$

In the $c$-small case the module $D^{\lambda,+}_{H_0,n}(M)$ admits an alternative description.

\begin{proposition}
\label{prop:D_n_product_decomposition}
Let $(G_0,\lambda)$ be a $c$-small pair and suppose we have chosen coordinates of $G_0$ such that $G_0 \cong \Gamma_0 \times H_0$ with $\Gamma_0$ corresponding to the 1st coordinate and $H_0$ corresponds to the other coordinates.

Then there is a natural isomorphism
$$D^{\lambda,+}_{H_0,n}(M)\xrightarrow{\cong} D^{\lambda,+}_{H_0,n}(\mathcal{C}^{\lambda \han}(H_0,M)^+) \widehat{\otimes}_{\Lambda^+_{H_0,n}} \mathcal{C}^{\lambda \han}(\Gamma_0,\Lambda_{H_0,n})^+.$$
\end{proposition}

\begin{proof}
Given $\un{k} \in \ZZ^d_{\geq 0}$ let ${\un{k}}_{\Gamma_0}$ denote the first coordinate of $\un{k}$ (corresponding to $\Gamma_0$) and let ${\un{k}}_{H_0}$ denote $\un{k}$ without the first coordinate (corresponding to $H_0$).
The description of the $\WLambda^+$-basis of $M^+_{\un{k},\lambda}$ shows there is an isomorphism
$$M^+_{\un{k},\lambda} \cong M^+ \otimes_{\WLambda^+} \mathcal{C}^{\lambda\han}(H_0,\WLambda^+)^{\Delta^{\un{k}_{H_0}}=0}\otimes_{\WLambda^+} \mathcal{C}^{\lambda\han}(
\Gamma_0,\WLambda^+)^{\Delta^{\un{k}_{\Gamma_0}}=0},$$
which can be rewritten as 
$$M^+_{\un{k},\lambda} \cong\mathcal{C}^{\lambda\han}(H_0,M^+)^{\Delta^{\un{k}_{H_0}}=0}\otimes_{\WLambda^+} \mathcal{C}^{\lambda\han}(
\Gamma_0,\WLambda^+)^{\Delta^{\un{k}_{\Gamma_0}}=0}.$$
Now one observes that $\mathcal{C}^{\lambda\han}(
\Gamma_0,\WLambda^+)^{\Delta^{\un{k}_{\Gamma_0}}=0}$ has a basis given by $\varpi^{\lfloor p^\lambda m \rfloor}\binom{x}{m}$ for $0 \leq m < \un{k}_{\Gamma_0}.$ It is $c$-fixed for the $\Gamma_0$-action because $(G_0,\lambda)$ is $c$-small (this is the same argument appearing in Lemma \ref{lem:small_action}). Hence by Proposition 4.10 of \cite{Po22b} there is a natural isomorphism
$$D_{H_0,n}^+(\mathcal{C}^{\lambda\han}(
\Gamma_0,\WLambda^+)^{\Delta^{\un{k}_{\Gamma_0}}=0}) \cong \mathcal{C}^{\lambda\han}(
\Gamma_0,\Lambda_{H_0,n}^+)^{\Delta^{\un{k}_{\Gamma_0}}=0}.$$

We then conclude by applying Lemma \ref{lem:D_n_tensor_prod}, taking the limit over $\un{k}$ and taking the completion.
\end{proof}

\subsection{Cohomology of $c$-small pairs}

Choose any $c > c_1 + 2c_2 + 2c_3$ and fix a $c$-small pair $(G_0,\lambda).$  In this subsection we shall simplify the $G_0$-cohomology of $\mathcal{C}^{\lambda\han}(G_0,M)$.

For every $\lambda$ and $n\geq n_0(M,\lambda)$ we have modules $D^{\lambda,+}_{H_0,n}(M^+)$ defined as in (\ref{def:big_Dplus}). In what follows, whenever we compute inside a cohomology group depending on $\lambda$ and involving $n$, we are going to implicitly assume $n \geq n_0(M,\lambda)$ so that $D^{\lambda,+}_{H_0,n}(M^+)$ is defined. We set $\Gamma_0 = G_0/H_0$, which is isomorphic to $\ZZ_p$ (for sufficiently small $G_0$).

\begin{proposition}
\label{prop:H_cohomology_vanishing}
For $i \geq 0$ we have natural isomorphisms $$\mathrm{H}^i(G_0,\mathcal{C}^{\lambda\han}(G_0,M)) \cong \mathrm{H}^i(\Gamma_0,\mathcal{C}^{\lambda\han}(G_0,M)^{H_0}) \cong \mathrm{H}^i(\Gamma_0,\WLambda^{H_0} \otimes_{\Lambda_{H_0,n}} D^{\lambda}_{H_0,n}(M)).$$
\end{proposition}

\begin{proof}
Inverting $\varpi$ in the isomorphism (\ref{iso:descent_of_funcs}), we have natural isomorphisms $$\WLambda \otimes_{\Lambda_{H_0,n}} D^{\lambda}_{H_0,n}(M) \cong \mathcal{C}^{\lambda\han}(G_0,M).$$ By Proposition 5.8 of \cite{Po24}, the cohomology $\mathrm{H}^i(H_0,\mathcal{C}^{\lambda\han}(G_0,M))$ is zero for $i \ge 1$. Note that in that setting one assumes that the valuation is $p$-adic, but the same proof works for a $\varpi$-adic valuation defined by an arbitrary topologically nilpotent unit $\varpi$. We conclude the proof by applying the Hochschild-Serre spectral sequence (see \cite[Lemma 3.3]{kedlaya2016hochschild} for a version which applies in our setting).
\end{proof}
Recall that we have maps $R_{H_0,n}: \WLambda^{H_0} \rightarrow \Lambda_{H_0,n}$ which are projections. Setting $X_{H_0,n} = \ker R_{H_0,n}$, we obtain a decomposition $\WLambda^{H_0} \cong X_{H_0,n} \oplus \Lambda_{H_0,n}$.
This allows us to write a $\Gamma_0$-equivariant decomposition
$$\WLambda^{H_0} \otimes_{\Lambda_{H_0,n}} D^{\lambda}_{H_0,n}(M) \cong X_{H_0,n} \otimes_{\Lambda_{H_0,n}} D^{\lambda}_{H_0,n}(M) \oplus D^{\lambda}_{H_0,n}(M)$$

and hence for $i \geq 1$
$$\mathrm{H}^i(\Gamma_0,\WLambda^{H_0} \otimes_{\Lambda_{H_0,n}} D^{\lambda}_{H_0,n}(M)) \cong \mathrm{H}^i(\Gamma,X_{H_0,n} \otimes_{\Lambda_{H_0,n}} D^{\lambda}_{H_0,n}) \oplus \mathrm{H}^i(\Gamma,D^{\lambda}_{H_0,n}(M)).$$

In particular, for $i \geq 2$ the cohomology is zero because $\Gamma_0 \cong \ZZ_p.$


\begin{proposition}
\label{prop:gamma_coh_vanishing_X}
We have $\mathrm{H}^1(\Gamma_0,X_{H_0,n}\otimes_{\Lambda_{H_0,n}} D^{\lambda}_{H_0,n}) = 0.$
\end{proposition}

\begin{proof}
In the above decomposition of the $\Gamma_0$-cohomology of $$\WLambda^{H_0} \otimes_{\Lambda_{H_0,n}} D^{\lambda}_{H_0,n}(M) \cong \mathcal{C}^{\lambda\han}(G_0,M)^{H_0}$$ the left hand side does not depend on $n$. Hence, the same holds for the right hand side. Therefore, it suffices to show that
$$\varinjlim_{n}\mathrm{H}^1(\Gamma_0,X_{H_0,n}\otimes_{\Lambda_{H_0,n}} D^{\lambda}_{H_0,n}) = 0.$$
Write $\gamma$ for a generator of $\Gamma_0$. It suffices to show that every $$a\otimes b \in X_{H_0,n} \otimes_{\Lambda_{H_0,n}} D_{H_0,n}^{\lambda}(M)$$ is in the image of $\gamma-1$ for some $X_{H_0,n'} \otimes_{\Lambda_{H_0,n'}} D_{H_0,n'}^{\lambda}(M)$, with $n'$ depending only on $n$.

By Proposition \ref{prop:action_uniformity}, we know that there exists some $s > 0$ such that 
$$(\gamma-1)(D_{H_0,n}^{\lambda,+}(M^+)) \subset \varpi^s\cdot D_{H_0,n}^{\lambda,+}(M^+).$$

For some power $\gamma'$ of $\gamma$ we therefore have  
$$(\gamma'-1)(D_{H_0,n}^{\lambda,+}(M^+)) \subset \varpi^{2c_2}\cdot D_{H_0,n}^{\lambda,+}(M^+).$$
Take $n' \geq n$ so that $n(\gamma') \leq n'$. For such $n'$, we know that $\val{(\gamma'-1)(x)}\geq \val(x)-c_2$ for $x \in X_{H_0,n'}$.

Let $\alpha$ be such that $(\gamma'^{-1}-1)(\alpha) = a$. Consider the series
$$y = \sum_{i=0}^\infty (\gamma'^{-1}-1)^{-i}(\alpha) \otimes (\gamma'-1)^{i}(b).$$
The series converges because $\val(\gamma'^{-1}-1)^{-1}(x) \geq \val(x) -c_3$ while $\val((\gamma'-1)(x))\geq \val(x) +2c_3$ by our choice of $n'$. A direct computation shows that $(\gamma'-1)(y) = a \otimes b$. As $\gamma'-1$ is divisible by $\gamma-1$, this proves $a \otimes b$ is in the image of $\gamma-1$, as required. 
\end{proof}
Putting this all together, we get 
\begin{theorem}
\label{thm:cohomology_for_c_small_pairs}
Let $c > c_1 + 2c_2+ 2c_3$ and let $(G_0,\lambda)$ be a $c$-small pair. Then $\mathrm{H}^i(G_0,\mathcal{C}^{\lambda\han}(G_0,M)) = 0$ for $i \geq 2$, and $$\mathrm{H}^1(G_0,\mathcal{C}^{\lambda\han}(G_0,M)) \cong \mathrm{H}^1(\Gamma_0,D_{H_0,n}^{\lambda}(M)).$$
\end{theorem}











\subsection{Vanishing of higher locally analytic vectors}

Recall that 
$$\R^i_{\la}(M) = \varinjlim_{G_0,\lambda} \mathrm{H}^i(G_0,\mathcal{C}^{\lambda \han}(G_0,M)).$$
In this subsection we shall complete the proof of Theorem \ref{main_thm}
by showing that $\R^i_{\la}(M) = 0$ for $i\geq 1$.
\begin{proof}
As always, we fix some $c > c_1 +2c_2 + 2c_3$. We claim that $c$-small pairs $(G_0,\lambda)$ are cofinal in the above direct limit. This is a consequence of Corollary \ref{cor:making_lambda_smaller} together with the observation that if $G_0$ is $c$-small, so is $G_l$. In particular, we have seen that $\mathrm{H}^i(G_0,\mathcal{C}^{\lambda \han}(G_0,M) = 0$ for $i\geq 2$ and $c$-small pairs $(G_0,\lambda)$, so it automatically follows that $\R^i_{\la}(M) = 0$ for $i \geq 2$. The nontrivial part is showing the vanishing of $\R^1_{\la}(M).$ Using the cofinality of $c$-small pairs, this is a consequence of the following proposition. 
\end{proof}

\begin{proposition}
\label{prop:killing_cocycle_coming_from_small_pairs}
Let $(G_0,\lambda)$ be a $c$-small pair and let $f\in \mathrm{H}^1(G_0,\mathcal{C}^{\lambda \han}(G_0,M))$. Then there exist some $l \geq 0$ and some $\lambda' \leq \lambda$ so that $f$ is mapped to $0$ in $\mathrm{H}^1(G_l,\mathcal{C}^{\lambda' \han}(G_l,M)).$
\end{proposition}

\begin{proof}
Recall that in \ref{thm:cohomology_for_c_small_pairs} we have shown an isomorphism for $n \geq n(M,\lambda)$
$$\mathrm{H}^1(G_0,\mathcal{C}^{\lambda \han}(G_0,M)) \cong \mathrm{H}^1(\Gamma_0,D_{H_0,n}(M)),$$

and so $f \in \mathrm{H}^1(\Gamma_0,D_{H_0,n}(M)).$ Using $\mathrm{H}^1(\Gamma_0,D_{H_0,n}(M)) \cong D_{H_0,n}(M)/(\gamma-1),$ we may think of $f$ as an element of $D_{H_0,n}(M)$.
Let $s = s(G_0,n)$ be the constant appearing in Proposition \ref{prop:action_uniformity}, and choose $\lambda'$ small enough that $p^{\lambda'} < s$. Choose $l \geq 0$ large enough so that 1 of Lemma \ref{lem:small_action} applies. We are going to show that $f$ is mapped to $0$ in $\mathrm{H}^1(G_l,\mathcal{C}^{\lambda' \han}(G_l,M)).$

By Proposition \ref{prop:D_n_product_decomposition} we may write $f$ as function (in the variable of $\Gamma_0$)
$$f(x) = \sum_{n \geq 0}{m_n}\binom{x}{n}$$
where $m_n \in D^{\lambda}_{H_0,n}(\mathcal{C}^{\lambda \han}(H_0,M))$ and $\val_\lambda(m_n) - \lfloor p^\lambda n \rfloor \rightarrow \infty.$ 

Consider now the function given by
$$F(x) = \sum_{k \geq 0} (-1)^k \sum_{n \geq 0} (\gamma-1)^k(\gamma^{-k}(m_n))\binom{x}{n+k+1}.$$
Recalling that $\gamma$ acts on $\binom{x}{n+k+1}$ by shifts, one easily checks that $(\gamma-1)(F) = f.$ The problem is that this sum is not guaranteed to converge in $D^{\lambda}_{H_0,n}(M)$. Indeed, what we have is
$$\val_\lambda((\gamma-1)^k(m_n)) \geq sk + \val_{\lambda}(m_n)$$
and
$$\val_\lambda(\binom{x}{n+k+1}) = -\lfloor p^\lambda (n+k+1)\rfloor$$
so we get the estimate
$$\val_\lambda((\gamma-1)^k(\gamma^{-k}(m_n))\binom{x}{n+k+1}) \geq k(s-p^\lambda) + (\val_{\lambda}(m_n) - \lfloor p^\lambda n \rfloor) + O(1).$$
This estimate would be good enough (i.e. tend to $\infty$ as $k, n \rightarrow \infty$) if we had $s > p^\lambda.$ Unfortunately, there is no reason why this would be the case. 

There is a trick which fixes this problem: the estimate is good enough when $\lambda$ is replaced with $\lambda',$ because $s > p^{\lambda'}$. In other words, by the same estimate we see that under the natural map 
$$D^\lambda_{H_0,n}(M) \rightarrow D^{\lambda'}_{H_l,n}(M)$$

the sum does converge (using $\val_{\lambda'} \geq \val_{\lambda}$). This means that the image of $f$ in $D^{\lambda'}_{H_l,n}(M)$ lies in the image of $\gamma-1$. 
Thus, decomposing as in \ref{prop:H_cohomology_vanishing} and the subsequent paragraph to get an isomorphism of $\mathrm{H}^1(G_0,\mathcal{C}^{\lambda' \han}(G_0,M))$ with

$$\mathrm{H}^1(G_0/H_l,X_{H_l,n}\otimes_{\Lambda_{H_l,n}} D^{\lambda'}_{H_l,n}(M)) \oplus \mathrm{H}^1(G_0/H_l,D^{\lambda'}_{H_l,n}(M)),$$
we obtain that under the natural map
$$\mathrm{H}^1(G_0,\mathcal{C}^{\lambda \han}(G_0,M))\rightarrow \mathrm{H}^1(G_0,\mathcal{C}^{\lambda' \han}(G_0,M))$$
the element $f$ lands in 
$$\ker(\mathrm{H}^1(G_0/H_l,D^{\lambda'}_{H_l,n}(M))\rightarrow \mathrm{H}^1(\Gamma_0,D^{\lambda'}_{H_l,n}(M))).$$

By the inflation restriction sequence, this is the same as $\mathrm{H}^1(H_0/H_l,D^{\lambda'}_{H_l,n}(M)^{\Gamma_0})$. Thus composing with the restriction map to $\mathrm{H}^1(G_l,\mathcal{C}^{\lambda' \han}(G_0,M))$ (and hence also with the further composition to $\mathrm{H}^1(G_l,\mathcal{C}^{\lambda' \han}(G_l,M))$) we see that $f$ is mapped to zero, as required!
\end{proof}

\begin{remark}
\label{rem:strong_local_analyticity} The proof shows that vanishing of $\R^{i}_{\la}(M)$ for $i\geq 1$ is true in the strong sense, namely, for each pair $(G_0, \lambda)$ in the direct limit $$R^i_{\la}(M) = \varinjlim_{G_0,\lambda} \mathrm{H}^i(G_0,\mathcal{C}^{\lambda \han}(G_0,M))$$
 there exists some other pair $(G_0',\lambda')$ so that the entire map $$\mathrm{H}^i(G_0,\mathcal{C}^{\lambda \han}(G_0,M)) \rightarrow \mathrm{H}^i(G_0',\mathcal{C}^{\lambda' \han}(G_0',M))$$ is zero.
\end{remark}

\section{Applications}

In this section we shall give three applications of our methods. For this we shall need to introduce a few objects which are standard in $p$-adic Hodge theory. For more details, we refer the reader to \cite{Co08} or $\mathsection{2.1}$ of \cite{Po22b}. Let $\mathbf{C}_p$ be the completion of the algebraic closue of $\QQ_p$. Let $\mathbf{C}_p^\flat$ be its tilt with ring of integers $\mathbf{C}_p^{\flat,+}$. Let $\A_{\inf}$ be the Witt vectors of $\mathbf{C}_p^{\flat,+}$. Let $\varepsilon = (\zeta_p,\zeta_p^2,...)$ be a sequence of compatible $p$-power roots of unity and let $\varpi = \varepsilon - 1 \in \mathbf{C}^{\flat,+}$ so that $[\varpi] \in \A_{\mathrm{inf}}$.
The ring $\widetilde{\mathbf{A}}^{(0,r],\circ}$ is defined to be the $p$-adic completion of $\A_{\mathrm{inf}}\langle p/[\varpi]^{1/r}\rangle$. The completion is $p$-adic, but we take the $[\varpi]$-adic topology. We define $\widetilde{\mathbf{A}}^{(0,r]}$ as $\widetilde{\mathbf{A}}^{(0,r],\circ}[1/[\varpi]]$ with the $[\varpi]$-adic topology.
We define also the ring $\widetilde{\mathbf{A}}^{[s,r]}:=\A_{\mathrm{inf}}\langle [\varpi]^{1/s}/p, p/[\varpi]^{1/r}\rangle,$ with the completion and topology being $p$-adic or $[\varpi]$-adic (they are the same) and $\widetilde{\mathbf{B}}^{[s,r]}:=\widetilde{\mathbf{A}}^{[s,r]}[1/p] = \widetilde{\mathbf{A}}^{[s,r]}[1/[\varpi]].$ Each ring $R$ just introduced has a continuous 
action of $\mathrm{Gal}(\overline{\QQ}_p/\QQ_p)$. Let $K$ be a finite extension of $\QQ_p$. Let $\QQ_p^{\mathrm{cyc}} = \QQ_p(\zeta_{p^\infty}),$ $K^{\cyc} = K\QQ_p^{\cyc}$ and $H_K = \mathrm{Gal}(\overline{\QQ}_p/K^{\cyc})$. We set $R_{K} := R^{H_K}$ so that $R_{K}$ 
has an action of $\mathrm{Gal}(\overline{\QQ}_p/K)/H_K \cong \mathrm{Gal}(K^{\cyc}/K)$. This latter group we call $\Gamma_K$ and it is isomorphic to an open subgroup of $\ZZ_p^\times$ via the cyclotomic character. With this notation we have $\widetilde{\mathbf{E}}_{K} := \widetilde{\mathbf{A}}^{(0,r]}_{K}/p$ (it does not depend on $r$). This field was mentioned in Example
\ref{example:loc_analytic_elems}.3 when $K=\QQ_p$. The field $\widetilde{\mathbf{E}}_{\QQ_p}$ can be equivalently defined as the $X$-adic completion of $\cup_n \FF_p((X^{1/p^n}))$.
The action of $\Gamma_{\QQ_p} = \ZZ_p^\times$ is given as follows: an element $a$ acts on $f(X) \in \widetilde{\mathbf{E}}_{\QQ_p}$ by the formula $(a\cdot f)(X) = f((1+X)^a -1).$ We also have "deperfected" rings defined as follows. We let $\mathbf{A}_{\QQ_p}$ be the $p$-adic completion of $\ZZ_p[[T]][1/T]$. It is naturally embedded into $\widetilde{\mathbf{A}} = W(\mathbf{C}_p^\flat)$ by mapping $T$ to the element $[\varepsilon]-1 \in \A_{\inf}$ which lifts $\varpi$. We have $\mathbf{A}_{\QQ_p}/p \cong \FF_p((X)) \subset \widetilde{\mathbf{E}}_{\QQ_p}.$ To each $K$ one can associate the field of norms $\mathbf{E}_K \subset \widetilde{\mathbf{E}}_K$ so that $\mathbf{E}_{\QQ_p} = \FF_p((X)),$ and there is a standard way to define rings $\mathbf{A}_{K}$ containing $\mathbf{A}_{\QQ_p}$ so that $\mathbf{A}_{K}/p \cong \mathbf{E}_K$ (see $\mathsection{6}$ of \cite{Co08}).
The ring $\widetilde{\mathbf{A}}^{(0,r],\circ}$ embeds into $\widetilde{\mathbf{A}}$ and one sets $\mathbf{A}^{(0,r],\circ}_{K} = \mathbf{A}^{(0,r],\circ} \cap \mathbf{A}_{K}$. We then let $\mathbf{A}^{(0,r]}_{K} = \mathbf{A}^{(0,r],\circ}_{K}[1/T]$. We let $\mathbf{A}^{[s,r]}_{\QQ_p}$ be the completion of the image of $\mathbf{A}^{(0,r]}_{K}$ in $\widetilde{\mathbf{A}}^{(0,r]}_{K}$. Finally, we set $\mathbf{B}^{(0,r]}_{K} = \mathbf{A}^{(0,r]}_{K}[1/p]$ and $\mathbf{B}^{[s,r]}_{K} = \mathbf{A}^{[s,r]}_{K}[1/p]$. In addition to the $\Gamma_K$-action above, these rings are endowed with a $\Gamma_K$-equivariant Frobenius operator $\varphi$ which maps rings defined via an interval $I$ to rings on the interval $p^{-1}I$. We let also $K^{\LT}$ denote a Lubin-Tate extension of $K$, $H_{\LT} = \mathrm{Gal}(\overline{\QQ}_p/K^{\LT})$ and $\Gamma_{\LT} = \mathrm{Gal}(\overline{\QQ}_p/K)/H_{\LT}.$

\begin{proposition}
\label{prop:tate_sen_axioms_verification}
The Tate-Sen axioms (TS1)-(TS4) of $\mathsection{3.2}$ are satisfied in the following cases: 

1. When $\WLambda = \widetilde{\mathbf{A}}^{(0,r]}_{\QQ_p}$, $\WLambda^+ = \widetilde{\mathbf{A}}^{(0,r],\circ}_{\QQ_p}$ and $\Lambda_{n} = \varphi^{-n}(\mathbf{A}^{(0,p^{-n}r]}_{\QQ_p})$ for $1/r \in \ZZ[1/p]_{\geq 0}$ with $r < 1$ (we omit subscripts $H$ in $\Lambda_{H,n}$ because $H_0 = 1$ in this case).

2.  When $\WLambda = \widetilde{\mathbf{A}}^{(0,r],H_{{LT}}}$, $\WLambda^+ = \widetilde{\mathbf{A}}^{(0,r],\circ}\cap \WLambda$ and $\Lambda_{H_L,n} = \varphi^{-n}(\mathbf{A}^{(0,p^{-n}r]}_{L})$ for $1/r \in \ZZ[1/p]_{\geq 0}$ with $r < 1$.
\end{proposition}

\begin{proof}
Case 1: the axioms (TS1)-(TS3) are verified in $\mathsection{5.1}$ of \cite{Po22b}. It suffices to show that for any $a\in 1+p\ZZ_p \subset \Gamma_{\QQ_p}$ and $k\in \ZZ$
we have $$\val((a-1)(\varphi^{-n}(T^k)) \geq \val(\varphi^{-n}(T^k)) + c$$ for some positive constant $c$.
Indeed, one computes (using the formula $a(T)=(1+T)^a-1)$) that
$$(a-1)(\varphi^{-n}(T^k)) = \varphi^{-n}(T^k)\cdot\varphi^{-n}((a+\sum_{m \geq 1}\binom{a}{m+1}T)^k-1).$$
As $a \equiv_p 1$ it follows that $$\val(\varphi^{-n}((a+\sum_{m \geq 1}\binom{a}{m+1}T)^k-1)) \geq \min(\val(p),\val(\varphi^{-n}(T))) > 0$$ as required.

Case 2: Again (TS2) and (TS3) are checked in \cite{Po22b}. (TS1) is shown in Lemma 10.1 of \cite{Co08}.
For (TS4), we argue as follows. The ring $\Lambda_{H_L,n}$ is endowed with the action of some finite index subgroup of $\Gamma_{\QQ_p}.$ It is finite free over $\Lambda_{n}$ so one may choose a 
$\Lambda_{n}$-basis $e_1,...,e_d.$ By possibly choosing an even smaller subgroup $\Gamma'$ of $\Gamma_{\QQ_p}$ one can arrange that $\Gamma' \subset 1 +p\ZZ_p$ and that the action of $\gamma'-1$ of a generator $\gamma'$ of $\Gamma'$ has $\val((\gamma'-1)(e_i) > c$ for some constant $c$. Now if $\gamma$ is a generator $1+p\ZZ_p$ then we already know $\val^{\op}_{\Lambda_n}(\gamma-1) \geq c'$ for some $c' > 0$. Since $\gamma-1$ divides $\gamma'-1$ we see (TS4) holds with $t = \min(c,c').$
\end{proof}

\subsection{Computation of locally analytic elements in $\widetilde{\mathbf{E}}_{\QQ_p}$}

Recall the main result of \cite{BR22}:

\begin{theorem}
\label{thm:E_decompletion}
We have $\widetilde{\mathbf{E}}^\la_{\QQ_p} = \cup_n \FF_p((X^{1/p^n})).$
\end{theorem}

We shall now explain how to derive this result by reducing rather formally to a previously known result of Berger in characteristic 0 which is proven with aid of $p$-adic analysis. Though as a whole the proof of \cite{BR22} is simpler than ours, the method we present here illustrates our hope of using this new technique of linking analytic vectors in characteristic 0 and characteristic $p$ in other contexts as well.

\begin{lemma}
\label{lem:intersected_period_rings}
We have $\widetilde{\mathbf{A}}^{(0,r]}_{\QQ_p} \cap \varphi^{-n}(\mathbf{B}^{[p^{-n}s,p^{-n}r]}_{\QQ_p}) = \varphi^{-n}(\mathbf{A}^{(0,p^{-n}r]}_{\QQ_p})$.
\end{lemma}

\begin{proof}
The inclusion of the right hand side in the left hand side is clear.
Conversely, suppose $x$ lies in the right hand side. By applying $\varphi^n$, we may reduce to the case $n = 0$. So we have $x \in \widetilde{\mathbf{A}}^{(0,r]}_{\QQ_p} \cap\mathbf{B}^{[s,r]}_{\QQ_p},$ whence $x$ lies in the larger ring $\widetilde{\mathbf{B}}^{(0,r]}_{\QQ_p} \cap\mathbf{B}^{[s,r]}_{\QQ_p} = \mathbf{B}^{(0,r]}_{\QQ_p}.$ Thus we reduce to showing $\mathbf{B}^{(0,r]}_{\QQ_p}/\mathbf{A}^{(0,r]}_{\QQ_p} \rightarrow \widetilde{\mathbf{B}}^{(0,r]}_{\QQ_p}/\widetilde{\mathbf{A}}^{(0,r]}_{\QQ_p}$ is injective. This reduces further to showing  $\widetilde{\mathbf{A}}^{(0,r]}/\mathbf{A}^{(0,r]}_{\QQ_p}$ is $p$-torsionfree. In fact it suffices to show $\widetilde{\mathbf{A}}^{(0,r],\circ}/\mathbf{A}^{(0,r],\circ}_{\QQ_p}$ is $p$-torsionfree, because $\widetilde{\mathbf{A}}^{(0,r]}/\mathbf{A}^{(0,r]}_{\QQ_p}$ is its $[\varpi]$-adic localization. 

Next, recall that $\mathbf{A}^{(0,r],\circ}_{\QQ_p} = \widetilde{\mathbf{A}}^{(0,r],\circ}\cap \widetilde{\mathbf{A}}_{\QQ_p}$, so $\widetilde{\mathbf{A}}^{(0,r]}/\mathbf{A}^{(0,r]}_{\QQ_p}$ injects into $\widetilde{\mathbf{A}}/\mathbf{A}_{\QQ_p}$. This will be $p$-torsionfree provided that $\mathbf{A}_{\QQ_p}\otimes {\FF_p} \rightarrow \widetilde{\mathbf{A}} \otimes {\FF_p}$ is injective. But this is easy: it is simply the map $\FF_p((X)) \rightarrow \mathbf{C}^\flat_p$.
\end{proof}

Recall Theorem 4.4 of \cite{Be16}.

\begin{theorem}
\label{thm:bergers_char_0_theorem}
We have $\widetilde{\mathbf{B}}^{[s,r],\la}_{\QQ_p} = \cup_n \varphi^{-n}(\mathbf{B}^{[p^{-n}s,p^{-n}r]}_{\QQ_p}).$
\end{theorem}

We have the following corollary which is the mixed characteristic version of Theorem \ref{thm:E_decompletion} and Theorem \ref{thm:bergers_char_0_theorem}.

\begin{corollary}
\label{cor:decompletion_integral_period_ring}
We have $\widetilde{\mathbf{A}}^{(0,r],\la}_{\QQ_p} = \cup_n \varphi^{-n}(\mathbf{A}^{(0,p^{-n}r]}_{\QQ_p}).$
\end{corollary}

\begin{proof}
Indeed, choosing some arbitrary $s$, we have
$$ \widetilde{\mathbf{A}}^{(0,r],\la}_{\QQ_p} \subset \widetilde{\mathbf{A}}^{(0,r]}_{\QQ_p} \cap \widetilde{\mathbf{B}}^{[s,r],\la}_{\QQ_p} = \widetilde{\mathbf{A}}^{(0,r]}_{\QQ_p} \cap (\cup_n \varphi^{-n}(\mathbf{B}^{[p^{-n}s,p^{-n}r]}_{\QQ_p})).$$ So by Lemma \ref{lem:intersected_period_rings}, the ring $ \widetilde{\mathbf{A}}^{(0,r],\la}_{\QQ_p}$ is contained in $\cup_n \varphi^{-n}(\mathbf{A}^{(0,p^{-n}r]}_{\QQ_p})$. Conversely, any element in $\cup_n \varphi^{-n}(\mathbf{A}^{(0,p^{-n}r]}_{\QQ_p})$ is locally analytic by Corollary \ref{cor:descent_loc_analytic}.
\end{proof}

From this we can deduce a new proof of Theorem \ref{thm:E_decompletion}, which after this set up becomes a one liner. Take some $0 < r < 1$ with $1/r \in \ZZ[1/p]$ and consider the short exact sequence
$$0 \rightarrow \widetilde{\mathbf{A}}^{(0,r]}_{\QQ_p}\rightarrow \widetilde{\mathbf{A}}^{(0,r]}_{\QQ_p} \rightarrow \widetilde{\mathbf{E}}_{\QQ_p} \rightarrow 0.$$

 By Theorem \ref{main_thm} and Proposition \ref{prop:tate_sen_axioms_verification} we have $\R^1_{\la}(\widetilde{\mathbf{A}}^{(0,r],\la}) = 0,$ so $\widetilde{\mathbf{E}}_{\QQ_p}^\la = \widetilde{\mathbf{A}}^{(0,r],\la}/p.$ Passing to locally analytic vectors, we get $$\widetilde{\mathbf{E}}_{\QQ_p}^\la = \cup_n \varphi^{-n}(\mathbf{A}^{(0,p^{-n}r]}_{\QQ_p})/p = \cup_n \FF_p((X^{1/p^n})).$$

\begin{remark}
\label{rem:mixed_characteristic_decompletion_from_mod_p_decompletion}
One can also argue in reverse and deduce Corollary \ref{cor:decompletion_integral_period_ring} from Theorem \ref{thm:E_decompletion}. It is not clear to us if one can push this further to deduce Theorem \ref{thm:bergers_char_0_theorem}.
\end{remark}

\subsection{Descent of Lubin-Tate $(\varphi,\Gamma)$-modules to locally analytic vectors}

Recall that according to Lubin-Tate theory the Lubin-Tate character gives an isomorphism $\Gamma_{\LT} \xrightarrow{\cong} \mathcal{O}_{K}^\times.$ The norm of the Lubin-Tate character is an unramified twist of the cyclotomic character, which shows that $K_{\LT}$ contains a twist $K^{\cyc,\eta}$ of the cyclotomic extension $K^{\cyc}$ by an unramified character $\eta$.

In the theory $p$-adic Galois representations and the $p$-adic Langlands program $(\varphi,\Gamma)$-modules have played a central role. In particular, the overconvergence theorem of Cherbonnier-Colmez (\cite{CC98}) has been a crucial component in providing a link between the Banach and the analytic sides of the Galois side for $\mathrm{GL}_2(\QQ_p).$ When $\QQ_p$ is replaced with a finite extension $K$, the situation with overconvergence is more complicated, as we shall now explain.

In this context one considers a big Robba ring $\widetilde{\mathbf{B}}_{\mathrm{rig},\LT}^\dagger$ and a "deperfected" Robba ring $\mathbf{B}_{\mathrm{rig},\LT}^\dagger$ which is a ring of power series in one variable converging in an annulus (see \cite{Be16}, where these rings are denoted $\widetilde{\mathbf{B}}_{\mathrm{rig},K}^\dagger$ and ${\mathbf{B}}_{\mathrm{rig},K}^\dagger$ respectively). Both of these rings are endowed with a Frobenius operator $\varphi$ and an action of $\Gamma_{\LT}.$ By the main result of Fourquaux-Xie (\cite{FX14}) it is known that when $K \neq \QQ_p$ there exist $(\varphi,\Gamma_{\LT})$-modules over $\widetilde{\mathbf{B}}_{\mathrm{rig},\LT}^\dagger$ which are not overconvergent, i.e. do not descend to $\mathbf{B}_{\mathrm{rig},\LT}^\dagger$. The main theorem of \cite{Be16} gives a sufficient condition for this to happen, namely, $K$-analyticity. However, in $\mathsection{8}$ of \cite{Be16} it is also shown that an arbitrary 
$(\varphi,\Gamma_{\LT})$-module over $\widetilde{\mathbf{B}}_{\mathrm{rig},\LT}^\dagger$ descends to the multivariable ring $\widetilde{\mathbf{B}}_{\mathrm{rig},\LT}^{\dagger,\mathrm{pa}}$ of pro-analytic vectors.  

The ring $\widetilde{\mathbf{B}}_{\mathrm{rig},K}^{\dagger,\mathrm{pa}}$ is a $\QQ_p$-algebra and as such it is a characteristic 0 object. On the other hand, when $K= \QQ_p$, Cherbonnier-Colmez show there is a descent to an integral ring $\mathbf{A}_{\QQ_p}^\dagger$ where $p$ is not invertible. We will now give an integral version of this descent when $K \neq \QQ_p.$ 

Let $\widetilde{\mathbf{A}}_{\LT} = \widetilde{\mathbf{A}}^{H_{\LT}}$ and $\widetilde{\mathbf{A}}^{\dagger}_{\LT} = \varinjlim_{r} \mathbf{A}^{(0,r]}_{\LT}$ where $\mathbf{A}^{(0,r]}_{\LT} := \mathbf{A}^{(0,r],H_{\LT}}.$ The idea is that the ring $\widetilde{\mathbf{A}}^{\dagger,\la}_{\LT}$ is an appropriate integral analogue of $\widetilde{\mathbf{B}}_{\mathrm{rig},K}^{\dagger,\mathrm{pa}}$.

\begin{theorem}
\label{thm:overconvergence_in_lubin_tate_case}

Every $(\varphi,\Gamma)$-module over $\widetilde{\mathbf{A}}^{\dagger}_{\LT}$ descends uniquely to a $(\varphi,\Gamma)$-module over $\widetilde{\mathbf{A}}^{\dagger,\la}_{\LT}.$
\end{theorem}

\begin{proof}
When $K^{\cyc} \subset K^{\LT}$ this follows from part 2 of Proposition \ref{prop:tate_sen_axioms_verification} and part 1 of Theorem \ref{main_thm}. In general, one can descend along an unramified twist (see $\mathsection{8}$ \cite{Be16} for a similar argument).
\end{proof}

\begin{remark}
Corollary \ref{cor:decompletion_integral_period_ring} shows that when $K = \QQ_p$ we recover the usual integral descent of Cherbonnier-Colmez.
\end{remark}

\subsection{Pro-analytic and locally analytic vectors in $\wt{\bf{A}}_{\QQ_p}$}
In this subsection we work out an analogy between $\mathrm{B}_{\mathrm{dR},\QQ_p}^{+}$ of $H_{\QQ_p}$ fixed-points of $\mathrm{B}_{\mathrm{dR}}^+$ and the ring $\wt{\bf{A}}_{\QQ_p}.$ 

Let $\{M_n\}_n$ be an inverse system of finite free modules over corresponding $\ZZ_p$-Tate algebras $\{R_n\}_n$ endowed with the action of a compact $p$-adic Lie group $G$. For the inverse limit $M = \varprojlim_n M_n$ we define
the pro-analytic vectors
$$M^{\pa} = \varprojlim_n M_n^{\la} = \varprojlim_n \varinjlim_{G_0,\lambda} M_n^{G_0,\lambda\han}$$
and the locally analytic vectors
$$M^{\la} = \varinjlim_{G_0,\lambda} \varprojlim_n M_n^{G_0,\lambda\han}.$$
Clearly there is a natural map $M^{\la} \rightarrow M^{\pa}$. Furthermore, it can be shown that when $R = \varprojlim_n R_n$ itself is a $\ZZ_p$-Tate algebra then $M^{\pa} = M^{\la}$ and that the new definition of locally analytic vectors agrees with the old one of $\mathsection{2}.$ For example, when the $M_n$ are $G$-Banach spaces, the module $M$ is a $G$-Frechet space and we recover the definitions appearing in \cite{BC16} and \cite{Be16}.

The ring $\mathrm{B}_{\mathrm{dR},{\QQ_p}}^{+}$ has the structure of a $\Gamma_{\QQ_p}$-Fréchet space. Indeed, letting $t$ denote Fontaine's element, one has
$$\mathrm{B}_{\mathrm{dR},\QQ_p}^{+} = \varprojlim_{n} \mathrm{B}_{\mathrm{dR},\QQ_p}^{+}/t^n.$$
It turns out there is a very nice description of the pro-analytic and locally-analytic elements in $\mathrm{B}_{\mathrm{dR},\QQ_p}^{+}$ (Proposition 2.6 of \cite{Po22a}).
\begin{theorem}
\label{thm:analytic_vecs_bdr}  
We have $\mathrm{B}_{\mathrm{dR},\QQ_p}^{+,\la} = \cup_n\QQ_p(\zeta_{p^n})[[t]]$ and $\mathrm{B}_{\mathrm{dR},\QQ_p}^{+,\pa} = \QQ_p^{\cyc}[[t]]$.
\end{theorem}

In a direct analogy with this we can consider the ring $\wt{\bf{A}}_{\QQ_p}$ as the inverse limit $$\wt{\bf{A}}_{\QQ_p} = \varprojlim_n \wt{\bf{A}}_{\QQ_p}/p^n,$$ with each $\wt{\bf{A}}_{\QQ_p}/p^n$ being a $\ZZ_p$-Tate algebra (with topologically nilpotent unit $T$). For example, we have $\wt{\bf{A}}_{\QQ_p}/p^n = \wt{\bf{E}}_{\QQ_p}.$ 

The third application of our main theorem is then the following result.

\begin{theorem}
\label{thm:analytic_elems_a_tilde}
We have $\wt{\bf{A}}_{\QQ_p}^{\la} = \varphi^{-\infty}({\bf{A}}_{\QQ_p})$ and $\wt{\bf{A}}_{\QQ_p}^{\pa} = \varprojlim_{n} \varphi^{-\infty}({\bf{A}_{{\QQ_p}}})/p^n.$ 
\end{theorem}

\begin{remark}
\label{rem:analytic_elements_a_tilde} The ring $\wt{\bf{A}}_{\QQ_p}^{\pa}$ can also be thought of as the subring of elements in $\wt{\bf{A}}_{\QQ_p}$ that for every $n\geq 1$  are congruent to an element of $\varphi^{-\infty}({\bf{A}_{{\QQ_p}}})$ modulo $p^n$.
This shows that the containment $\wt{\bf{A}}_{\QQ_p}^{\la} \subset \wt{\bf{A}}_{\QQ_p}^{\pa}$ is strict. For example, the power series $\sum_{n \geq 0} p^n \varphi^{-n}(1+T)$ belongs to $\wt{\bf{A}}_{\QQ_p}^{\pa}$ but not to $\wt{\bf{A}}_{\QQ_p}^{\la}$.
\end{remark}

\begin{remark}
\label{rem:lifting_field_of_norms} Let It has been desirable to have a theory of $(\varphi,\Gamma)$ modules when $\Gamma = \Gal(\overline{K}/K)/H$ is not Lubin-Tate. However, in such cases it seems difficult to find a suitable lift of the field of norms, that is, a replacement for the ring ${\bf{A}}_K$. A lift that is a power series ring inside $\wt{\bf{A}}^{H}$ is known not to exist under certain assumptions  (\cite{Be14}, \cite{Poy22}). When $\Gamma$ is abelian, Theorem \ref{thm:analytic_elems_a_tilde} suggests to consider $\wt{\bf{A}}^{H,\la}$ or $\wt{\bf{A}}^{H,\pa}$ as a suitable lift instead. Unpublished computations of ours suggest that in general $\wt{\bf{A}}^{H,\pa}$ surjects onto $\wt{\bf{E}}^{H,\la}$ and so gives a valid lift.
\end{remark}

We now proceed to proving Theorem \ref{thm:analytic_elems_a_tilde}. The description of $\wt{\bf{A}}_{\QQ_p}^{\pa}$ follows directly from the following lemma.

\begin{lemma}
\label{lem:analytic_elements_tilde_A_mod_pn}
Let $0 < r < 1$. We have
$$(\wt{\bf{A}}_{\QQ_p}/p^n)^\la = (\cup_m\varphi^{-m}({\bf{A}}_{\QQ_p}^{(0,p^{-m}r]}))/p^n =\varphi^{-\infty}({\bf{A}}_{\QQ_p})/p^n.$$
\end{lemma}

\begin{proof}
We have $$\wt{{\bf{A}}}_{\QQ_p}^{(0,r]}/p^n = \wt{{\bf{A}}}_{\QQ_p}/p^n$$ by devissage to the case $n=1$ (where both are equal to $\wt{\bf{E}}_{\QQ_p}$).
By Theorem \ref{main_thm} we have $\R^1_{\la}(p^{n-1}\wt{{\bf{A}}}^{(0,r]}_{\QQ_p}/p^n) = 0$. Hence by devissage we deduce
$$(\wt{{\bf{A}}}_{\QQ_p}^{(0,r]}/p^n)^\la = (\cup_m\varphi^{-m}({\bf{A}}_{\QQ_p}^{(0,p^{-m}r]}))/p^n$$
(the case $n=1$ being Theorem \ref{thm:E_decompletion}).
Finally, we need to explain why 
$$(\cup_m\varphi^{-m}({\bf{A}}_{\QQ_p}^{(0,p^{-m}r]}))/p^n =\varphi^{-\infty}({\bf{A}}_{\QQ_p})/p^n,$$
but this is once again true by devissage.
\end{proof}

We now turn to studying $\wt{\bf{A}}_{\QQ_p}^\la.$
\begin{lemma}
\label{lem:control_of_analyticity_a_tilde}
1. Given $\Gamma \subset \Gamma_{\QQ_p}$ open, $\lambda \in \RR$ and $r < 1$ there exists an $m \geq 0$ such that $$(\wt{{\bf{A}}}_{\QQ_p}^{(0,r]})^{\Gamma,\lambda\han} \subset \varphi^{-m}({\bf{A}}_{\QQ_p}^{(0,rp^{-m}]}).$$

2. Given $m \geq 0$ and $r < 1$ there exists $\Gamma \subset \Gamma_{\QQ_p}$ open and $\lambda \in \RR$ such that $$\varphi^{-m}({\bf{A}}_{\QQ_p}^{(0,rp^{-m}]})\subset (\wt{{\bf{A}}}_{\QQ_p}^{(0,r]})^{\Gamma,\lambda\han}.$$
\end{lemma}
\begin{proof}
Part 1 follows from Theorem 4.4 of \cite{Be16} and Lemma \ref{lem:intersected_period_rings}. Part 2 follows from Corollary \ref{cor:descent_loc_analytic}.
\end{proof}

\begin{proposition}
\label{prop:control_of_analyticity_a_tilde_mod_p_powers}
1. For $\Gamma \subset \Gamma_{\QQ_p}$ open and $\lambda \in \RR$ there exists some $m \geq 0$ so that
for all $n \geq 1$ there is an inclusion
$$(\wt{{\bf{A}}}_{\QQ_p}^{(0,r]}/p^n)^{\Gamma,\lambda \han} \subset \varphi^{-m}({\bf{A}}_{\QQ_p}^{(0,rp^{-m}]})/p^n.$$
2. Given $m \geq 0$ there exists $\Gamma \subset \Gamma_{\QQ_p}$ open and $\lambda \in \RR$ so that for all $n \geq 0$ there is an inclusion
$$\varphi^{-m}({\bf{A}}_{\QQ_p}^{(0,rp^{-m}]})/p^n \subset (\wt{{\bf{A}}}_{\QQ_p}^{(0,r]}/p^n)^{\Gamma,\lambda \han}.$$
\end{proposition}
\begin{proof}
We start by proving part 2. There is a surjection $$\varphi^{-m}({\bf{A}}_{\QQ_p}^{(0,rp^{-m}]}) \twoheadrightarrow \varphi^{-m}({\bf{A}}_{\QQ_p}^{(0,rp^{-m}]})/p^n.$$ Hence by part 2 of the previous lemma and functoriality, there exists some pair $\Gamma, \lambda$, independent of $n$, such that every element of $\varphi^{-m}({\bf{A}}_{\QQ_p}^{(0,rp^{-m}]})/p^n$ lies in $(\wt{{\bf{A}}}_{\QQ_p}^{(0,r]}/p^n)^{\Gamma,\lambda \han}$.
The proof of part 1 is similar but we need another trick, because we do not know that $(\wt{{\bf{A}}}_{\QQ_p}^{(0,r]})^{\Gamma,\lambda \han}$ surjects onto $(\wt{{\bf{A}}}_{\QQ_p}^{(0,r]}/p^n)^{\Gamma,\lambda \han}$. Rather, we note that by Remark \ref{rem:strong_local_analyticity}, the entirety of $\R^1_{\Gamma,\lambda}(\wt{{\bf{A}}}_{\QQ_p}^{(0,r]})$ maps to zero in some $\R^1_{\Gamma',\lambda'}(\wt{{\bf{A}}}_{\QQ_p}^{(0,r]})$. This implies that given a pair $(\Gamma,\lambda)$ there exist some other pair $(\Gamma',\lambda')$ so that for all $n\geq 1$ we have that elements of $(\wt{\bf{A}}_{\QQ_p}^{(0,r]}/p^n)^{\Gamma,\lambda \han}$ lift to $(\wt{\bf{A}}_{\QQ_p}^{(0,r]})^{\Gamma',\lambda' \han}$. We can now argue as before by applying part 1 of the previous lemma (for the pair $\Gamma',\lambda'$).
\end{proof}

\begin{corollary}
    We have $\wt{{\bf{A}}}_{\QQ_p}^{\la} = \varphi^{-\infty}({\bf{A}}_{\QQ_p}).$
\end{corollary} 

\begin{proof}
The proposition shows that the direct systems
$$\{\varprojlim_{n}({\bf{A}}_{\QQ_p}^{(0,r]}/p^n)^{\Gamma,\lambda \han}\}_{\Gamma,\lambda}$$and $$\{\varprojlim_{n}\varphi^{-m}({\bf{A}}_{\QQ_p}^{(0,rp^{-m}]})/p^n \}_m$$ are cofinal. So the corresponding direct limits ranging over $\Gamma,\lambda$ and over $m$ are naturally isomorphic. These give ${\bf{A}}_{\QQ_p}^\la$ and $\varphi^{-\infty}({\bf{A}}_{\QQ_p})$ respectively and hence we establish the desired equality.
\end{proof}

\bibliography{main}
\end{document}